\documentclass[a4paper]{article}
\usepackage{geometry}
\usepackage{graphicx,url}
\usepackage[margin=10pt, font=small, labelfont=bf]{caption}

\usepackage[english]{babel}   
\usepackage[T1]{fontenc}
\usepackage[utf8]{inputenc}

\usepackage{slashed}
\usepackage{mathtools}
\mathtoolsset{showonlyrefs}

\usepackage{tikz}
\usepackage{verbatim}
\usepackage{listings}
\usepackage{xcolor}
\usepackage{amsmath,amssymb,mathrsfs}
\usepackage{mathtools}
\usepackage{amsthm}
\usepackage{frontespizio}
\usepackage{pictexwd, dcpic}
\usepackage{epigraph}
\usepackage{enumerate}
\usepackage{fancyhdr}
\usepackage{textcomp}
\usepackage{cancel}
\usepackage{bbm}
\usepackage{authblk}
\theoremstyle{plain}
\newtheorem{thm}{Theorem}[subsection]
\newtheorem{prop}[thm]{Proposition}

\newtheorem{lem}[thm]{Lemma}
\theoremstyle{definition}

\newtheorem{rem} [thm] {Remark}

\title{On the Cauchy problem for the  reaction-diffusion system with point-interaction in $\mathbb R^2$}
\date{}
\author[1]{Daniele Barbera}
\author[2,3,4]{Vladimir Georgiev}
\author[2]{Mario Rastrelli}

\affil[1]{Dipartimento di Scienze Matematiche
"Giuseppe Luigi Lagrange" \\
Politecnico di Torino \\
Corso Duca degli Abruzzi, 24
10129 Torino, Italy}  
\affil[2]{Dipartimento di Matematica  \\ Universit\`a di Pisa \\ Largo B. Pontecorvo 5, 56100 Pisa, Italy}
\affil[3]{Faculty of Science and Engineering \\ Waseda University \\
 3-4-1, Okubo, Shinjuku-ku, Tokyo 169-8555 \\
Japan}
\affil[4]{Institute of Mathematics and Informatics, Bulgarian Academy of Sciences, Acad.
Georgi Bonchev Str., Block 8, 1113 Sofia, Bulgaria}

\begin{document}

\maketitle

\abstract{The paper studies the existence of solutions for the reaction-diffusion equation in $\mathbb R^2$ with point-interaction laplacian $\Delta_\alpha$ with $\alpha\in(-\infty,+\infty]$, assuming the functions to remain on the absolute continuous projection space. By semigroup estimates, we get the existence and uniqueness of solutions on 
$$ L^\infty\left((0,T);H^1_\alpha\left(\mathbb R^2\right)\right)\cap L^r\left((0,T);H^{s+1}_\alpha\left(\mathbb R^2\right)\right), $$
with $r>2$, $s<\frac{2}{r}$ for the Cauchy problem with small $T>0$ or small initial conditions on $H^1_\alpha(\mathbb R^2)$. Finally, we prove decay in time of the functions.}

\vspace{2mm}

\textbf{Keywords:} Point-interaction laplacian, Reaction-diffusion equation, semigroup estimates

\vspace{2mm}

\textbf{2020 Mathematics Subject Classification: } 35K57, 35K45, 35A21

\section{Introduction}

\subsection{Physical and Mathematical Background}

Let us consider the system:
\begin{equation}\label{m.sys.0}
\left\{\begin{array}{ll}
     (\partial_t-\Delta_\alpha)u = a\cdot\nabla (|u|^\gamma) & (0,T)\times \mathbb R^2 \\
     u(0)=u_0 & \mathbb R^2,
\end{array}\right. 
\end{equation}
for some $T>0$, $a\in\mathbb R^2$ and $\gamma>1$. Here $u\colon (0,T)\times\mathbb R^2\to \mathbb R$, $u_0\colon \mathbb R^2\to \mathbb R$ and $\Delta_\alpha$ is the one-parameter family of self-adjoint extensions of the operator $\Delta_{|C^\infty_c(\mathbb R^2\setminus\{0\})}$ depending on $\alpha\in(-\infty,+\infty]$ (see Subsection \ref{subsec.Delta-alpha} for more details).

The standard convection-diffusion model, namely
\begin{equation}\label{cdm1}
\left\{\begin{array}{ll}
     (\partial_t-\Delta)u = a\cdot\nabla (|u|^\gamma) & (0,T)\times \mathbb R^2 \\
     u(0)=u_0 & \mathbb R^2,
\end{array}\right. 
\end{equation}
with $T>0$, $a\in\mathbb R^2$ and $\gamma>1$, is largely study to describe the morphogenesis of living beings, i.e. the process of shape formation of cells and tissues (see \cite{T52}, \cite{KM10}, \cite{BVE11}, \cite{MR12}). Moreover, it has been recently noticed the development of reaction-diffusion processes at nano-scales (see \cite{E16}). 

The goal of the present paper is to study diffusion models with point interactions. Point interactions are used to model the behaviour of electrons in quantum dots, which are semiconductor particles that confine electrons in three dimensions. These models help to understand the electronic properties and energy levels of quantum dots, which are crucial for applications in quantum computing and optoelectronics (see \cite{K20}). Moreover, point interactions can describe the diffusion of atoms or vacancies in a crystal lattice, particularly near defects or impurities. This is important for understanding the mechanical and thermal properties of materials, as well as for processes like doping in semiconductors (see \cite{Sir11}).

Rigorous mathematical introduction of point interaction was started with pioneering works \cite{AH81}, \cite{ABD95}, \cite{AGKH05}, where point interactions are modelled (in dimension $N=1,2,3$) as appropriate self - adjoint extensions of the Laplace operator defined on $C_c^\infty(\mathbb{R}^N\setminus \{0\})$. Recently, the main contributions in this field have been related with Schr\"odinger equations. This is, we have
\begin{enumerate}
    \item[a)] Study of perturbed Sobolev spaces $H^s_\alpha$ associated with the Laplace operator $\Delta_\alpha$ with point interaction (see \cite{GMR18}, \cite{GMS24}),
    \item[b)] Proof of existence of wave operators associated with linear Schr\"odinger group $e^{i\Delta_\alpha t},$ (see \cite{CMJ20}, \cite{CMJ20ER}, \cite{DMSY18}, \cite{CFN23}). Existence of local and global solutions to corresponding non-linear Schr\"odinger problem with local or non-local interaction (\cite{CFN23}, \cite{FGI22}, \cite{GR25}) . Blow up phenomena are studied in \cite{FN23}, \cite{Forcella2024},
    \item[c)] Analysis of ground states associated with non-linear Schr\"odinger equations (see \cite{FGI22}, \cite{GMS22}, \cite{GMS24}),
    \item[e)] Study of perturbed Sobolev spaces $H^{s,p}_\alpha$ associated with the Laplace operator $\Delta_\alpha$ with point interaction (see \cite{GR25}, \cite{GeorgievRastrelli2023}).
\end{enumerate}
In the study of point-type interactions for diffusion models, we have to take into account the fact that the operator $\Delta_\alpha$ has always point spectrum in dimension 2. 

In the case of the pure Laplace operator, the standard convection - diffusion equation \eqref{cdm1} is studied in \cite{EZ91}, \cite{Zuazua1993}, \cite{Zuazua2020}, \cite{Ku24}. The existence of Fujita-type exponents and asymptotic behaviour of the solution for large tine is studied in \cite{Zuazua1993}, \cite{Zuazua2020}, \cite{Ku24}. However, the presence of $\Delta_\alpha$ in the system \eqref{m.sys.0} makes some strategies of the cited papers hard to apply in our case. As an example, the global existence result is obtained in \cite{EZ91} by exploiting a maximum principle, which does not hold for the corresponding parabolic semigroup $e^{\Delta_\alpha t}$. Note that, even in the case
\begin{equation*}
    (\partial_t-\Delta+V(x))u = a\cdot\nabla \left(|u|^\gamma\right) ,
\end{equation*}
there is no deep understanding how the small data global existence results for the case $V=0$, can be extended to the case when
$-\Delta +V(x)$ has simple point spectrum. These observations show that much less is known about the convection diffusion equation associated with $\Delta_\alpha$ in dimension 2. It turns out to be a challenging problem that involves some new phenomena and results that are expected.

\subsection{Laplace with point interaction}\label{subsec.Delta-alpha}

Before stating the main theorems of the paper, we devote this subsection to explain the basic properties of the operators $\Delta_\alpha$, with $\alpha\in(-\infty,\infty]$ in dimension $N=2$ following the work of \cite{AGKH05}.

If we consider the operator $\Delta_{|C^\infty_c(\mathbb R^2\setminus\{0\})}$, then we can parametrize with $\alpha\in \mathbb{R}$ its non trivial self-adjoint extensions $\Delta_\alpha$. When $\alpha=\infty$, $\Delta_\alpha$ corresponds to the standard Laplacian, otherwise  the structure of the domain and the action are well-known.
If we fix $\alpha\in \mathbb{R}$ and consider $\lambda>0$ sufficiently big, then we can characterize the domain
\begin{equation}\label{dom.Lap.alpha}
    \mathcal{D}(\Delta_\alpha)=\left\{u\in L^2\left(\mathbb{R}^2\right)\:\Bigg|\:u=\phi_\lambda+\frac{\phi_\lambda(0)}{\alpha+c(\lambda)}G_\lambda, \ \phi_\lambda\in H^2\left(\mathbb{R}^2\right)\right\}
\end{equation}
and the action
\begin{equation}\label{eq.Domain}
    (\lambda-\Delta_\alpha)u=(\lambda-\Delta)\phi_\lambda.
\end{equation}
Here, $G_\lambda$ is the $L^2$-solution of  the Helmholtz equation with Dirac delta $(\lambda-\Delta)G_\lambda=\delta_0$, i.e.
\begin{equation}\label{eq.action}
    G_\lambda(x)=\frac{1}{2\pi}K_0\left(\sqrt{\lambda}|x|\right),
\end{equation}
where $K_0$ is the modified Bessel function of the second kind of order zero, and $c(\lambda)$ is a constant that represents its behaviour near zero, i.e.
    \begin{equation}
        c(\lambda)= \frac{\gamma-\ln{2}}{2\pi}+\frac{1}{2\pi}\ln{\sqrt{\lambda}},
    \end{equation}
with $\gamma\simeq0.577$ the Euler-Mascheroni constant. If we call with 
\begin{equation}\label{eigenvalue}
    E_\alpha= 4 e^{-4\pi \alpha -2\gamma}
\end{equation}
the solution of $\alpha+c(E_\alpha)=0,$ it can be proven that the domain \eqref{eq.Domain} and the action \eqref{eq.action} are independent on $\lambda\neq E_\alpha$. 
Otherwise $E_\alpha$ is the unique positive eigenvalue of $\Delta_\alpha$, with normalized eigenfunction
\begin{equation}\label{eigenfuntion}
    \psi_\alpha=\frac{G_{E_\alpha}}{\|G_{E_\alpha}\|_{L^2(\mathbb{R}^2)}}.
\end{equation}
Finally, the structure of the spectrum $\sigma(\Delta_\alpha)$ and also the resolvent formula are well-known: we have 
\begin{equation}
\begin{aligned}
    &\sigma(\Delta_\alpha)=\sigma_{ess}(\Delta_\alpha)\cup\sigma_p(\Delta_\alpha),&\begin{aligned}
        \sigma_{ess}(\Delta_\alpha)=(-\infty,0],\\
        \sigma_p(\Delta_\alpha)=\{E_\alpha\},
    \end{aligned}
\end{aligned}
\end{equation}
and, by the Krein's approach, 
\begin{equation}\label{eq.resolvent formula}
    (\lambda-\Delta_\alpha)^{-1}g=(\lambda-\Delta)^{-1}g+\frac{\langle g,G_\lambda\rangle}{\alpha+c(\lambda)}G_\lambda, 
\end{equation}
for every $\lambda\in \mathbb {R}^+\backslash\{E_\alpha\}$. Moreover, we can decompose $L^2$ with the two orthogonal projections 
\begin{equation}\label{proj.def.}
    \begin{aligned}
        &\begin{aligned}
            P_d:L^2(\mathbb{R}^N)&\to\operatorname{Span}\{\psi_\alpha\}\\
            \varphi&\mapsto\langle\varphi,\psi_\alpha\rangle_{L^2}\psi_\alpha,
        \end{aligned}
        &\begin{aligned}
            P_{ac}:L^2(\mathbb{R}^N)&\to\operatorname{Span}\{\psi_\alpha\}^{\bot}\\
            P_{ac}&=I-P_d.
        \end{aligned}
    \end{aligned}
\end{equation}
By Cauchy-Schwarz inequality, it is immediate to see that they are bounded. Moreover, $P_d$ is symmetric, so it is self-adjoint, and obviously commutes with the identity $I$, so $P_{ac}$ is self-adjoint too. 

Thanks to H\"older inequality and density argument, these projections can be defined on $L^p(\mathbb R^2)$ for every $p\in(1,\infty)$. So, for any $g\in L^p(\mathbb R^2)$, orthogonality gives  
$$ \|P_dg\|_{L^p(\mathbb R^2)}\le \|g\|_{L^p(\mathbb R^2)}, $$
$$ \|P_{ac}g\|_{L^p(\mathbb R^2)}\le \|g\|_{L^p(\mathbb R^2)}. $$
\begin{rem}\label{rem.Lap-alpha.dom.}
In \cite{GMS24}, fractional domains are studied. We write 
$$ H^s_\alpha\left(\mathbb R^N\right)=D\left(\left(\lambda-\Delta_\alpha\right)^{s/2}\right), $$ for $s\in[0,2]$ and $\lambda>E_\alpha$, with corresponding norm
\begin{equation}
    \|u\|_{H^s_\alpha(\mathbb{R}^2)}=\left\|(\lambda-\Delta_\alpha)^{s/2}u\right\|_{L^2(\mathbb{R}^2)}.
\end{equation}. More precisely, it can be seen the following characterization
    \begin{itemize}
        \item When $0\leq s<1$, then $H^s_\alpha(\mathbb R^2)=H^s(\mathbb R^2)$ for any $\alpha\in\mathbb R$ and the corresponding norms are equivalent;
        \item When  $s=1$, then
        $$ H^1_\alpha\left(\mathbb R^2\right)= \left\{u\in L^2\left(\mathbb R^2\right)\mid u=\phi_\lambda + cG_\lambda,\ \  \phi_\lambda\in H^1\left(\mathbb{R}^2\right), \:\:c\in\mathbb{C}\right\} $$
       and $$\|\phi_\lambda+cG_\lambda\|_{H^1_\alpha(\mathbb{R}^2)}^2\approx\|\phi_\lambda\|_{H^1(\mathbb{R}^2)}^2+|c|^2. $$
        \item When $1<s\leq2$, then
        $$ H^s_\alpha\left(\mathbb R^2\right)= \left\{u\in L^2\left(\mathbb R^2\right)\:\Bigg|\: u=\phi_\lambda + \frac{\phi_\lambda(0)}{\alpha+c(\lambda)}G_\lambda,\ \  \phi_\lambda\in H^s\left(\mathbb{R}^2\right)\right\}. $$
    \end{itemize}   

\end{rem}

For $N=3$ the definitions are similar: the domain 
\begin{equation}
    \mathcal{D}(\Delta_\alpha)=\left\{u\in L^2\left(\mathbb{R}^3\right)\:\Bigg|\:u=\phi_\lambda+\frac{\phi_\lambda(0)}{\alpha+c(\lambda)}G_\lambda, \ \phi_\lambda\in H^2\left(\mathbb{R}^3\right)\right\}
\end{equation}
and the action
\begin{equation}
    (\lambda-\Delta_\alpha)u=(\lambda-\Delta)\phi_\lambda
\end{equation}
are identical. However, in this case 
\begin{equation}
    G_\lambda(x)=\frac{e^{-\sqrt{\lambda}|x|}}{4\pi|x|}
\end{equation}
and 
\begin{equation}
    c(\lambda)=\frac{\sqrt{\lambda}}{4\pi}.
\end{equation}
The eigenvalue exists if and only if $\alpha<0$, it is positive and equal to 
\begin{equation}
    E_\alpha=(4\pi\alpha)^2.
\end{equation}
\begin{rem}\label{rem.Sob.emb.}
We notice that, due to Lemma 2.2 of \cite{GMS24}, the standard Sobolev embedding for $H^1(\mathbb R^2)$ still holds for $H^1_\alpha(\mathbb R^2)$, that is:
$$ H^1_\alpha\left(\mathbb R^2\right)\hookrightarrow L^q\left(\mathbb R^2\right) \quad \forall q\in[2,\infty). $$
\end{rem}

\subsection{Main results}

The aim of the paper is to find local and global solutions for the system 
\begin{equation}\label{m.sys.}
\left\{\begin{array}{ll}
     (\partial_t-\Delta_\alpha)u = a\cdot\nabla (|u|^\gamma) & (0,T)\times \mathbb R^2 \\
     u(0)=u_0 & \mathbb R^2,
\end{array}\right. 
\end{equation}
for some $T>0$, $a\in\mathbb R^2$ and $\gamma>1$. In particular, we want to find solutions in the energy space as in \cite{BG24} and \cite{BG25}, that is 
$$ L^\infty\left((0,T); H^1_\alpha\left(\mathbb R^2\right)\right)\cap L^2\left((0,T);H^2_\alpha\left(\mathbb R^2\right)\right),  $$
where we defined the spaces $H^s_\alpha$ in Remark \ref{rem.Lap-alpha.dom.}. Let us state the local existence result:
\begin{thm}\label{t.ex.loc.main}
Let $\alpha\in\mathbb R$, $a\in\mathbb R^2$, $\gamma >1$, let $u_0\in H^1_\alpha(\mathbb R^2)$, then there is $T>0$ such that the system \eqref{m.sys.} admits a unique solution $u$ such that
$$ u\in L^\infty\left((0,T); H^1_\alpha\left(\mathbb R^2\right)\right)\cap L^r\left((0,T);H^{s+1}_\alpha\left(\mathbb R^2\right)\right)\cap L^p\left((0,T);L^q\left(\mathbb R^2\right)\right) $$
for any  $r>2,$ $s\in\left(0,\frac{2}{r}\right)$, $p\ge 1$ and $q\in[2,\infty)$.
\end{thm}

Here below we present some difficulties and new tools used in establishing this result.

\begin{enumerate}
    \item[i)] We note that the lack of maximum principle for the linear  equation
    \begin{equation}
        (\partial_t - P_{ac}\Delta_\alpha)u =0
    \end{equation}
    is an essential obstacle to obtain a local existence result with initial data in space of type $L^\infty.$ In fact, we can not apply the maximum principle argument as in \cite{EZ91}. Note   that the inclusion $H^s_\alpha(\mathbb{R}^2) \subseteq L^\infty(\mathbb{R}^2)$ for $s \in (1,2]$ is not valid in the perturbed Sobolev space. However, the possibility to apply appropriate decay estimates for Laplace operator with point interaction is an important new tool to overcome these difficulties.
    \item[ii)]  A standard application of the energy method for \eqref{m.sys.} meets some difficulties, due to the decomposition in \eqref{dom.Lap.alpha}. For this reason, we choose a solution space in a form weaker than the classical one for the  free Laplacian:
$$ L^r\left(\mathbb R_+;H^{s_r+1}_\alpha\left(\mathbb R^2\right)\right)\subseteq L^\infty\left(\mathbb R_+;H^1_\alpha\left(\mathbb R^2\right)\right)\cap L^2\left(\mathbb R_+;H^2_\alpha\left(\mathbb R^2\right)\right) \ \ \forall r>2, $$
where $s_r=\frac{2}{r}$. In fact, we use semigroup estimates to bound directly the terms of the Duhamel Formula corresponding to the solution for \eqref{m.sys.}, loosing the endpoint estimate in $L^2((0,T);H^2_\alpha(\mathbb R^2))$. 
\end{enumerate}

Next,  we turn to the global existence problem. This is even harder: the operator $\Delta_\alpha$ admits a positive eigenvalue so the semigroup $e^{\Delta_\alpha t}u_0$ has an exponential growth in time. For this reason, we consider the system 
\begin{equation}\label{m.sys.proj.}
    \left\{\begin{array}{ll}
    (\partial_t- P_{ac}\Delta_\alpha)u= P_{ac}(a\cdot\nabla)(|u|^\gamma) & \mathbb R_+\times\mathbb R^2 \\
    u(0)=P_{ac}u_0 & \mathbb R^2,
\end{array}\right. 
\end{equation}
where the projection $P_{ac}$ is defined in \eqref{proj.def.}. 
\begin{thm}\label{t.ex.gl.m.}
Let $\alpha\in\mathbb R$, $a\in\mathbb R^2$ and $\gamma >1$, then there is $\varepsilon_0>0$ such that, for any $\varepsilon\le \varepsilon_0$ and $u_0\in H^1_\alpha(\mathbb R^2)\cap L^1(\mathbb R^2)$ with
$$ \|u_0\|_{H^1_\alpha\cap L^1(\mathbb R^2)}\le \varepsilon, $$
the system \eqref{m.sys.proj.} admits a unique solution $u$ such that
$$ u\in L^\infty\left(\mathbb R_+; H^1_\alpha\left(\mathbb R^2\right)\right)\cap L^r\left(\mathbb R_+;H^{s+1}_\alpha\left(\mathbb R^2\right)\right)\cap L^p\left(\mathbb R_+;L^q\left(\mathbb R^2\right)\right) $$
for any  $r>2,$ $s\in\left(0,\frac{2}{r}\right)$, and for any $p\ge 1$ and $q\in[2,\infty)$ with $\frac{1}{p}+\frac{1}{q}<1$.
\end{thm}
If we want to consider the problem \eqref{m.sys.proj.} without the projection $P_{ac}$, we can introduce a function $\rho\colon\mathbb R_+\to\mathbb R$ and study the system 
\begin{equation}\label{m.sys.proj.2}
    \left\{\begin{array}{ll}
    (\partial_t- \Delta_\alpha)u + \rho(t)\psi_\alpha=a\cdot\nabla(|u|^\gamma) & \mathbb R_+\times\mathbb R^2 \\
    u(0)=P_{ac}u_0 & \mathbb R^2,
\end{array}\right. 
\end{equation}
where $\psi_\alpha$ is defined in \eqref{eigenfuntion}. Then we have the following result:

\begin{thm}\label{t.ex.gl.m.2}
Let $\alpha\in\mathbb R$, $a\in\mathbb R^2$ and $\gamma >1$, then there is $\varepsilon_0>0$ such that, for any $\varepsilon\le \varepsilon_0$ and $u_0\in H^1_\alpha(\mathbb R^2)\cap L^1(\mathbb R^2)$ with
$$ \|u_0\|_{H^1_\alpha\cap L^1(\mathbb R^2)}\le \varepsilon, $$
the system \eqref{m.sys.proj.2} admits a unique solution $(u,\rho)$ such that
$$ u\in L^\infty\left(\mathbb R_+; H^1_\alpha\left(\mathbb R^2\right)\right)\cap L^r\left(\mathbb R_+;H^{s+1}_\alpha\left(\mathbb R^2\right)\right)\cap L^p\left(\mathbb R_+;L^q\left(\mathbb R^2\right)\right),\quad \rho\in L^1\cap L^r\left(\mathbb R_+\right), $$
for any $r>2$, $s\in\left(0,\frac{2}{r}\right)$, and for any $p\ge 1$ and $q\in[2,\infty)$ with $\frac{1}{p}+\frac{1}{q}<1$. In particular,
$$ \rho(t)=\left<a\cdot \nabla(|u|^\gamma),\psi_\alpha\right>_{L^2(\mathbb R^2)}. $$
\end{thm}
The presence of $\rho=\rho(u)$ can be seen also from another point of view: let $u$ be a solution of \eqref{m.sys.proj.2}, then we show that the orthogonality condition
$$ \langle u(t), \psi_\alpha  \rangle_{L^2} $$
is preserved under the nonlinear flow defined in 
$$ (\partial_t- \Delta_\alpha)u + \rho(u)(t)\psi_\alpha=a\cdot\nabla(|u|^\gamma)  \mathbb R_+\times\mathbb R^2. $$
Therefore, $(u,\rho)$ solves
$$ \left\{\begin{array}{ll}
    (\partial_t- \Delta_\alpha)u + \rho(t)\psi_\alpha=a\cdot\nabla(|u|^\gamma) & \mathbb R_+\times\mathbb R^2 \\
    P_{ac}u=u & \mathbb R_+\times\mathbb R^2 \\
    u(0)=P_{ac}u_0 & \mathbb R^2.
\end{array}\right. $$
Therefore, $\rho(t)$ is the Lagrange multiplier associated with the constraint $P_{ac}u=u$. Similarly, any solution $u$ of \eqref{m.sys.proj.} satisfies the property $P_{ac}u=u$. So, to proved the global existence for our problem, we consider the systems \eqref{m.sys.proj.} and \eqref{m.sys.proj.2} to neglect the part of the solution associated with the positive eigenvalue of $\Delta_\alpha$. 

To conclude, we can prove a polynomial decay in time of the solutions:
\begin{thm}\label{t.decay.m.}
Let $(u,\rho)$ be the solution from Theorem \ref{t.ex.gl.m.2}, then for any $h_1\in(1,\infty)$ and $h_2\in(1,2)$ for which there are $\theta_1,\theta_2\in\left(\frac{1}{2},1\right)$ with
$$ \theta_1+\theta_2>\frac{3}{2},\quad \max\left\{\frac{1}{h_1},\frac{1}{h_2}\right\}<\frac{\theta_1}{h_1}+\frac{\theta_2}{h_2}+\frac{1-\theta_2}{2}<1, $$
there is $\delta_0>0$ such that, for any $t\ge 1$ and for any $\delta\le \delta_0$, it holds
$$ \|u(t)\|_{L^{h_1}(\mathbb R^2)}\lesssim t^{-1+\frac{1}{h_1}+\delta}, $$
$$ \|\nabla u(t)\|_{L^{h_2}(\mathbb R^2)}\lesssim t^{-\frac{3}{2}+\frac{1}{h_2}+\delta}, $$
$$ |\rho(t)|\lesssim t^{-1-\delta}. $$
\end{thm}
We notice that the set of exponents that satisfy the previous statement is not empty: if $h_1,h_2$ are such that
$$ \frac{1}{h_1}+\frac{1}{h_2}<1, \quad h_1\in(1,\infty),\quad h_2\in(1,2), $$
then they are admissible for Theorem \ref{t.decay.m.}.
\begin{rem}
In system \eqref{m.sys.}, we considered a function $u\colon (0,T)\times\mathbb R^2\to\mathbb R$. Applying the same argument, it is possible to study the equation
$$ (\partial_t-\Delta_\alpha)u=A\cdot\nabla\left(|u|^\gamma\right) \quad  (0,T)\times\mathbb R^2, $$
for $u\colon (0,T)\times\mathbb R^2\to\mathbb R^m$ and $A\in L^\infty((0,T)\times\mathbb R^2; \mathbb R^m\times\mathbb R^2)$, with $m=1,2$.
\end{rem}

The structure of the paper is the following: in Section 2, we focus on the a priori estimates corresponding to the linearization of the main system \eqref{m.sys.}. In Section 3, we focus on the local existence and on the proof of Theorem \ref{t.ex.loc.main}. Finally, in Section 4 we study the global existence and subsequently the decay in time of such a solutions. 

\vspace{2mm}

\textbf{Notation:} In the following, we denote for simplicity
$$ L^p_TV\coloneqq L^p((0,T);V) $$
and
$$ L^pV\coloneqq L^p(\mathbb R_+;V) $$
for $p\in[1,\infty]$, and $V=H^s_\alpha(\mathbb R^2), H^s(\mathbb R^2)$ with $s\ge 0$.

\section{Linear Estimates}
\subsection{Semigroup Estimates}

The aim of this subsection is to extend the classical power-type decay 
 \begin{equation}\label{eq.classic time decay}
 \|e^{\Delta t}g\|_{L^p(\mathbb R^N)}\lesssim t^{-\frac{N}{2}\left(\frac{1}{q}-\frac{1}{p}\right)}\|g\|_{L^q(\mathbb R^N)} 
\end{equation}
for some $g$ to the operator $e^{\Delta_\alpha t}$ with $N=2,3$. It is immediate to see that the existence of the eigenfunction \eqref{eigenfuntion} does not allow an estimate like \eqref{eq.classic time decay} in all $L^2(\mathbb{R}^N)$, because
$$e^{\Delta_\alpha t}\psi_\alpha=e^{E_\alpha t}\psi_\alpha,$$
so we can have an exponential growth. 
\begin{prop}\label{p.proj.comm.} 
The operator $\Delta_\alpha$ commutes with the projections $P_d$ and $P_{ac}$ defined in \eqref{proj.def.}:
\begin{equation}
    \begin{aligned}
        &\Delta_\alpha P_{d}=P_{d}\Delta_\alpha, &  &\Delta_\alpha P_{ac}=P_{ac}\Delta_\alpha.
    \end{aligned}
\end{equation}
\end{prop}
\begin{proof}
    It is sufficient to prove the first equality.
    For every $\varphi\in D(\Delta_\alpha)$, we have
    \begin{equation}
        \begin{aligned}
            \Delta_\alpha P_{d}\varphi&=\Delta_\alpha\langle\varphi,\psi_\alpha\rangle_{L^2}\psi_\alpha\\
            &=\langle\varphi,\psi_\alpha\rangle_{L^2}E_\alpha\psi_\alpha\\
            &=\langle\varphi,E_\alpha\psi_\alpha\rangle_{L^2}\psi_\alpha\\
            &=\langle\varphi,\Delta_\alpha\psi_\alpha\rangle_{L^2}\psi_\alpha\\
            &=\langle\Delta_\alpha \varphi,\psi_\alpha\rangle_{L^2}\psi_\alpha=P_d\Delta_\alpha\varphi,\\ 
        \end{aligned}
    \end{equation}
    where the last line follows from self-adjointness.
\end{proof}
\begin{rem}
    The projections are self-adjoint and commute with $\Delta_\alpha$ which is also self-adjoint, so they are simultaneously diagonalizable. 
\end{rem}
We see that these two projections are very useful to avoid the eigenvalue. For this reason, we cite two very classical results (see for example \cite{RS78}, section XII)
\begin{thm}
   Let $A$ self-adjoint operator with domain $D(A)$. Let $M\subseteq D(A)$ a closed subset such that $A(M)\subseteq M$, then
   \begin{enumerate}
       \item $D(A)\cap M^\bot$ is dense in $M^\bot$;
       \item $h\in D(A)\cap M^\bot\implies Ah\in M^\bot $;
       \item $AP_M:M\to M$ is bounded on $M$ and 
       $AP_{M^\bot}:D(A)\cap M^\bot\to M^\bot$ is self-adjoint on $M^\bot$.
   \end{enumerate}
\end{thm}
In our case $\Delta_\alpha$ $M=Ker(E_\alpha-\Delta_\alpha)$,  $P_M=P_d$, $P_{M^\bot}=P_{ac}$, and we can characterize the spectrum of $\Delta_\alpha P_{ac}$.
\begin{thm}
    Let $A$ self-adjoint operator and $\lambda_0$ eigenvalue. Then $\lambda$ is an isolated point of $\sigma(A)$
   if and only if
    $$\lambda_0\notin\sigma\left(AP_{Ker(\lambda_0-A)^\bot}\right),$$
    or, equivalently ,
    $$\rho(A)=\rho\left(AP_{Ker(\lambda_0-A)^\bot}\right)\backslash\{\lambda_0\}.$$
\end{thm}
These results provide the following characterization
$$\sigma(\Delta_\alpha P_{ac})=\sigma_{ess}(\Delta_\alpha )=(-\infty,0].$$
We can also describe the exponential and the resolvent of $\Delta_\alpha P_{ac}$ in terms of $\Delta_\alpha$.
\begin{prop}\label{prop.eq res=}
    We have the following identities
    \begin{equation}\label{eq.exp=}
        e^{t\Delta_\alpha P_{ac}}P_{ac}g=e^{t\Delta_\alpha }P_{ac}g,
    \end{equation}
    \begin{equation}\label{proj.id.1}
     P_{ac} e^{P_{ac}\Delta_\alpha t}g=e^{P_{ac}\Delta_\alpha t}P_{ac} g,
\end{equation}
    and
    \begin{equation}\label{eq.res=}
        (\lambda-\Delta_\alpha P_{ac})^{-1}P_{ac}g=(\lambda-\Delta_\alpha)^{-1}P_{ac}g
    \end{equation}
    for every $g\in D(\Delta_\alpha)$ and $\lambda\in \mathbb{C}\backslash(-\infty,0]$.
\end{prop}
\begin{proof}
    First of all, we note that $u(t)=e^{t\Delta_\alpha P_{ac}}P_{ac}g$ is well defined for every $g\in D(\Delta_\alpha)$, because the exponent is a self-adjoint operator. Moreover, $u$ is defined as the unique solution of the following Cauchy problem
    \begin{equation}\label{eq.Cauchy1}
        \left\{\begin{aligned}
            &\partial_t u-\Delta_\alpha P_{ac}u=0, \\
            &u(0)=P_{ac}g.
        \end{aligned}\right.
    \end{equation}
    Applying $P_d$ in both equations, we obtain $P_d u(t)\equiv0$, because $P_d$ commutes with $\Delta_\alpha$ and the linear operator $\partial_t$. This means that $P_{ac}u(t)=u(t)$.
    So, if $u$ is a solution of \eqref{eq.Cauchy1}, it solves also 
    \begin{equation}\label{eq.Cauchy2}
        \left\{\begin{aligned}
            &\partial_t u-\Delta_\alpha u=0, \\
            &u(0)=P_{ac}g.
        \end{aligned}\right.
    \end{equation}
    Since \eqref{eq.Cauchy2} has a unique solution, this implies \eqref{eq.exp=}.
    Similarly, \eqref{proj.id.1} can be proven. This is because $P_{ac}e^{P_ac\Delta_\alpha t}g=P_{ac}f$, 
    with $f$ solution of
    \begin{equation}\label{eq.Cauchy3}
        \left\{\begin{aligned}
            &\partial_t f-P_{ac}\Delta_\alpha f=0, \\
            &f(0)=g.
        \end{aligned}\right.
    \end{equation}
    Applying $P_{ac}$ to both parts in \eqref{eq.Cauchy3}, we see that $P_{ac}$ solves \eqref{eq.Cauchy2}.
    To prove \eqref{eq.res=}, the procedure is very similar. 
    We consider $u=(\lambda-\Delta_\alpha P_{ac})^{-1}P_{ac}g$, that is equivalent to
   \begin{equation}\label{eq.rs}
       (\lambda-\Delta_\alpha P_{ac})u=P_{ac}g
   \end{equation}
    and we apply $P_d$ to both terms. This provides $\lambda P_du=0$, with $\lambda\neq0$. We have again $P_{ac}u=u$, that allows us to obtain \eqref{eq.res=}, thanks to \eqref{eq.rs}. 
\end{proof}

\begin{thm}\label{t.expDelta alpha}  Let $N=2,3$ and $1<q < p <\infty$ for $N=2$ or $\frac{3}{2}<q < p <3$ for $N=3$. There exists a constant $C$, independent of $t$, such that the following inequality holds for every $t>0$
    \begin{equation}\label{eq.exp timedecay1}
    \left\|e^{\Delta_\alpha t}P_{ac}g\right\|_{L^p(\mathbb R^N)}\leq C t^{-\frac{N}{2}\left(\frac{1}{q}-\frac{1}{p}\right)}\left\|P_{ac}g\right\|_{L^q(\mathbb R^N)}.
    \end{equation}
\end{thm}
\begin{proof}
    We remember that, if $A$ is a sectorial operator, then we have a representation of the semigroup through the Laplace transform. In particular, the following formula holds
\begin{equation}\label{eq.Dunford integral}
    e^{tA}=\frac{1}{2\pi i}\int _{\Gamma} e^{t\lambda}(\lambda-A)^{-1}d\lambda, \ \ t>0,
\end{equation}
    for every $\Gamma$ curve that surrounds the spectrum of $A$.

 If we consider the resolvent formula \eqref{eq.resolvent formula} for $\Delta_\alpha$, it can be extended to the whole set of resolvents $\rho(\Delta_\alpha)$, considering the principal value of the complex logarithm $\operatorname{Log}\sqrt{\lambda}=\ln{\sqrt{|\lambda|}}+\frac{i}{2}\operatorname{arg}(\lambda)$.

With \eqref{eq.Dunford integral} and Proposition \ref{prop.eq res=} wee see that
\begin{equation}\label{eq.exponentialDeltaalpha1}
    e^{\Delta_\alpha t}P_{ac}g=\frac{1}{2\pi i}\int_{\Gamma}e^{\lambda t}(\lambda-\Delta_\alpha)^{-1}P_{ac}gd\lambda,
\end{equation}
for every $\Gamma$ that surrounds $\sigma(\Delta_\alpha P_{ac})=(-\infty,0]$.
With \eqref{eq.resolvent formula}, we reduce \eqref{eq.exponentialDeltaalpha1} to
\begin{equation}\label{eq.exponentialDeltaalpha2}
    e^{t\Delta_\alpha}P_{ac}g=e^{t\Delta}P_{ac}g+\frac{1}{2\pi i}\int _{\Gamma} e^{t\lambda} \frac{\langle P_{ac}g,G_\lambda\rangle}{\alpha+c(\lambda)}G_\lambda d\lambda
\end{equation}
because $\sigma(\Delta)= \sigma (\Delta_\alpha P_{ac})$. The formula \eqref{eq.exponentialDeltaalpha2} is independent from $\Gamma$, so  we choose the curve $\Gamma=\Gamma_1\cup \Gamma_2\cup \Gamma_3$, with
\begin{equation}
    \begin{aligned}
       &\Gamma_1=\{s-\varepsilon i|s\leq 0\},\\
        &\Gamma_2=\{\varepsilon e^{is}|s\in[-\pi/2,\pi/2]\},\\
        &\Gamma_3=\{-s+\varepsilon i|s\geq0\},
    \end{aligned}   
\end{equation} 
with $\varepsilon<E_\alpha$.
We have the following rescaling property
\begin{equation}\label{eq.rescaling G}
    G_\lambda(|x|)= \lambda^{\frac{N}{2}-1}G_1(\sqrt{\lambda}|x|),
\end{equation}
that implies
\begin{equation}
    \|G_\lambda\|_{L^p(\mathbb R^N)}=|\lambda|^{\frac{N}{2}-1-\frac{N}{2p}}\|G_1\|_{L^p(\mathbb R^N)}
\end{equation}
and, thanks to Holder inequality, 
\begin{equation}\label{eq.Holder G}
    |\langle P_{ac}g,G_\lambda\rangle|\leq |\lambda|^{\frac{N}{2q}-1}\|P_{ac}g\|_{L^q(\mathbb R^N)}\|G_1\|_{L^{q^\prime}(\mathbb R^N)}.
\end{equation}
Here we used the hypothesis $p\neq\infty$ and $q\neq 1$ for $N=2$ and $p<3$ and $q>\frac{3}{2}$ for $N=3$. Thanks to the fact that $\Gamma$  is far from $E_\alpha$, we have 
\begin{equation}\label{eq.clambda}
    \begin{aligned}
       & \left|\alpha+{\frac{\gamma-\ln{2}}{2\pi}}+\frac{1}{2\pi}\ln\sqrt{|\lambda|}+\frac{i}{2}\operatorname{arg}(\lambda)\right|\gtrsim 1,\\
        &\left|\alpha+\frac{\sqrt{\lambda}}{4\pi}\right|\gtrsim |\lambda|^{\frac{1}{2}},
    \end{aligned}
\end{equation}
so we can estimate the difference 
\begin{equation}\label{eq.exp diff}
\begin{aligned}
\|e^{t\Delta_\alpha}P_{ac}g-e^{t\Delta}P_{ac}g\|_{L^p(\mathbb R^N)}&\lesssim\int _{\Gamma} |e^{t\lambda}|{|\lambda|^{\frac{N}{2q}-1-\frac{N}{2p}}}|d\lambda| \|P_{ac}g\|_{L^q(\mathbb R^N)}\\
&=(I_1(t)+I_2(t)+I_3(t))\|P_{ac}g\|_{L^q(\mathbb R^N)},
\end{aligned}
\end{equation}
where we have denoted with
\begin{equation}
    I_j(t)=\int _{\Gamma_j} |e^{t\lambda}|{|\lambda|^{\frac{N}{2q}-1-\frac{N}{2p}}}|d\lambda|.
\end{equation}
For the first integral, we compute
\begin{equation}
    I_1(t)=\int_{-\infty}^{0}e^{ts}|s-\varepsilon i|^{\frac{N}{2q}-1-\frac{N}{2p}} ds= \int_{-\infty}^{0}e^{ts}(s^2+\varepsilon^2)^{\frac{1}{2}\Big(\frac{N}{2}(\frac{1}{q}-\frac{1}{p})-1\Big)}ds,
\end{equation}
 and with the change of variable $ts=\sigma$ we obtain
\begin{equation}\label{eq.I1}
    I_1(t)= \int_{-\infty}^{0}e^\sigma |\sigma^2+t^2\varepsilon^2|^{\frac{1}{2}\left(\frac{N}{2}(\frac{1}{q}-\frac{1}{p})-1\right)}d\sigma \ t^{-\frac{N}{2}\left(\frac{1}{q}-\frac{1}{p}\right)}.
\end{equation}
Because $p>q>\frac{N}{2}$, then $-1<\frac{N}{2}\left(\frac{1}{q}-\frac{1}{p}\right)-1<0$, so we can estimate
\begin{equation}
    I_1(t)\leq t^{-\frac{N}{2}\left(\frac{1}{q}-\frac{1}{p}\right)}\int_{-\infty}^{0}e^\sigma |\sigma|^{\frac{N}{2q}-\frac{N}{2p}-1}d\sigma\leq C t^{-\frac{N}{2}\left(\frac{1}{q}-\frac{1}{p}\right)}. 
\end{equation}

The computation of $I_3(t)$ is very similar:
\begin{equation}\label{eq.I3}
    I_3(t)\leq\int_0^\infty e^{-ts}|\varepsilon i-s|^{\frac{N}{2q}-1-\frac{N}{2p}}ds\leq C t^{-\frac{N}{2}\left(\frac{1}{q}-\frac{1}{p}\right)}.
\end{equation}
We see that $I_2(t)$, thanks to symmetry, is zero
\begin{equation}\label{eq.I2}
    I_2(t)=\int_{-\frac{\pi}{2}}^{\frac{\pi}{2}}e^{t\varepsilon \sin{s}}\varepsilon^2 sds\leq \varepsilon^2e^{t\varepsilon}\int_{-\frac{\pi}{2}}^\frac{\pi}{2} sds=0.
\end{equation}
Combining \eqref{eq.I1}, \eqref{eq.I2}, \eqref{eq.I3} and \eqref{eq.classic time decay} in \eqref{eq.exp diff}, we have the thesis.
\end{proof}

The projection $P_{ac}$ on the absolute continuous space can be removed from \eqref{eq.exp timedecay1} considering $t\in(0,T)$ for some $T>0$.
\begin{thm}\label{t.loc.sem.es.}\hfill\\
Let $N=2,3$, $T>0$ and $1<q < p <\infty$ for $N=2$ or $\frac{3}{2}<q < p <3$ for $N=3$. There exists a constant $C(T)$, continuous for $T\ge0$, such that the following inequality holds for every $t\in[0,T]$
    $$ \left\|e^{\Delta_\alpha t}g\right\|_{L^p(\mathbb R^N)}\leq C(T) t^{-\frac{N}{2}\left(\frac{1}{q}-\frac{1}{p}\right)}\|g\|_{L^q(\mathbb R^N)}. $$
\end{thm}
\begin{proof}
     The proof of this theorem is very similar to the one of Theorem \ref{t.expDelta alpha}, so we will explain only the main points. We use again the Dunford integral defined in \eqref{eq.Dunford integral} for $\Delta_\alpha$, but this time the curve $\Gamma$ surrounds the whole spectrum $\sigma(\Delta_\alpha)=(-\infty,0]\cup\{E_\alpha\}$.
We choose $\Gamma=\Gamma_1\cup \Gamma_2\cup \Gamma_3$, with
\begin{equation}
    \begin{aligned}
       &\Gamma_1=\{s- i|s\leq E_\alpha\},\\
        &\Gamma_2=\{E_\alpha+ e^{is}|s\in[-\pi/2,\pi/2]\},\\
        &\Gamma_3=\{-s+ i|s\geq E_\alpha\}.
    \end{aligned}   
\end{equation} 
We obtain again
$$ e^{t\Delta_\alpha}g=e^{t\Delta}g+\frac{1}{2\pi i}\int _{\Gamma} e^{t\lambda} \frac{\langle g,G_\lambda\rangle}{\alpha+c(\lambda)}G_\lambda d\lambda, $$
also because $\sigma(\Delta)\subseteq \sigma(\Delta_\alpha)$, so $\Gamma$ surrounds also the spectrum of the free Laplacian.
With \eqref{eq.rescaling G}, \eqref{eq.Holder G} and \eqref{eq.clambda}, the following estimate can be easily obtained

\begin{equation}
\begin{aligned}
\|e^{t\Delta_\alpha}g-e^{t\Delta}g\|_{L^p(\mathbb R^N)}&\lesssim(I_1(t)+I_2(t)+I_3(t))\|g\|_{L^q(\mathbb R^N)},
\end{aligned}
\end{equation}
where
\begin{equation}
    I_j(t)=\int _{\Gamma_j} |e^{t\lambda}|{|\lambda|^{\frac{N}{2q}-1-\frac{N}{2p}}}|d\lambda|.
\end{equation}
It is easy to compute:
\begin{equation}
\begin{aligned}
     &I_j(t)\leq \int_{-\infty}^{tE_\alpha}e^\sigma |\sigma|^{\frac{N}{2q}-1-\frac{N}{2p}}d\sigma \ t^{-\frac{N}{2}\left(\frac{1}{q}-\frac{1}{p}\right)}\leq Ce^{TE_\alpha}t^{-\frac{N}{2}\left(\frac{1}{q}-\frac{1}{p}\right)}, &j=1,3
\end{aligned}
\end{equation}
and
\begin{equation}
    I_2(t)\leq e^{tE_\alpha}\int_{-\frac{\pi}{2}}^\frac{\pi}{2}e^{t\sin{s}}sds\leq0.
\end{equation}

\end{proof}

\begin{rem}\label{rem.L2.sem.es.}
    To be noticed that the previous estimate does not work for $p=q$. Anyway, when $p=q=2$, it follows by standard functional analysis argument that 
    $$ \left\|e^{\Delta_\alpha t}g\right\|_{L^2(\mathbb R^N)}\le C(T)\|g\|_{L^2(\mathbb R^N)}\quad \forall t\in(0,T), $$
    for some $C(T)>0$ continuous for $T\ge 0$.
\end{rem}

For $N=2$, we can also extend these results to the gradient. This is not possible in $N=3$, because 
$$\nabla G_\lambda\in L^p(\mathbb{R}^3)\iff p<\frac{3}{2},$$
but 
$$ G_\lambda\in L^{q^\prime}(\mathbb{R}^3)\iff q>\frac{3}{2}. $$
So, it is not possible to have a range where $q<p$ and the condition for $p $ and $q$ above are satisfied.
\begin{thm}\label{t.sem-grad.es.}\hfill
    Let $N=2$ and $1<q<p<2$. There exists a constant $C$ independent of $t$, such that the following inequality holds for every $t\geq0$
    $$ \left\|\nabla e^{\Delta_\alpha t}P_{ac}g\right\|_{L^p(\mathbb R^2)}\leq C t^{-\frac{1}{2}-\left(\frac{1}{q}-\frac{1}{p}\right)}\|P_{ac}g\|_{L^q(\mathbb R^2)}. $$
\end{thm}
\begin{proof}
    From \eqref{eq.exponentialDeltaalpha2}, we can write 
    \begin{equation}
        \nabla e^{\Delta_\alpha t}P_{ac}g- \nabla e^{\Delta t}P_{ac}g=\frac{1}{2\pi i}\int _\Gamma e^{t\lambda}\frac{\langle P_{ac}g,G_\lambda\rangle}{\alpha+c(\lambda)}\nabla G_\lambda d\lambda,
    \end{equation}
for every curve $\Gamma$ that surrounds $\sigma(\Delta_\alpha P_{ac})=(-\infty,0]$.
As done in the proof of Theorem \ref{t.loc.sem.es.}, we consider the curve $\Gamma=\Gamma_1\cup \Gamma_2\cup \Gamma_3$, with
\begin{equation}
    \begin{aligned}
       &\Gamma_1=\{s-\varepsilon i|s\leq 0\},\\
        &\Gamma_2=\{\varepsilon e^{is}|s\in[-\pi/2,\pi/2]\},\\
        &\Gamma_3=\{-s+\varepsilon i|s\geq0\},
    \end{aligned}   
\end{equation} 
with $\varepsilon<E_\alpha$.
The rescaling property \eqref{eq.rescaling G} can be extended to the gradient 
\begin{equation}
    \nabla G_\lambda(|x|)=\lambda^\frac{1}{2}\nabla G_1(\sqrt{\lambda}|x|),
\end{equation}
that gives
\begin{equation}\label{eq.rescaling nabla G}
    \|\nabla G_\lambda\|_{L^p}(\mathbb{R}^2)=|\lambda|^{\frac{1}{2}-\frac{1}{p}}\|\nabla G_1\|_{L^p(\mathbb{R}^2)}.
\end{equation}
With \eqref{eq.rescaling nabla G}, \eqref{eq.Holder G} and \eqref{eq.clambda}, we can estimate:
\begin{equation}
    \begin{aligned}
        \|\nabla e^{\Delta_\alpha t}P_{ac}g- \nabla e^{\Delta t}P_{ac}g\|_{L^p(\mathbb{R}^2)}&\lesssim \int_\Gamma|e^{t\lambda}| |\lambda|^{\frac{1}{q}-1}|\lambda|^{\frac{1}{2}-\frac{1}{p}}|d\lambda|\|P_{ac}g\|_{L^q(\mathbb R^2)}\\
        &\lesssim\int_\Gamma|e^{t\lambda}| |\lambda|^{\frac{1}{q}-\frac{1}{2}-\frac{1}{p}}|d\lambda|\|P_{ac}g\|_{L^q(\mathbb R^2)}\\
        &=(I_1(t)+I_2(t)+I_3(t))\|P_{ac}g\|_{L^q(\mathbb R^2)},
    \end{aligned}
\end{equation}
with 

   \begin{equation}
    I_j(t)=\int _{\Gamma_j} |e^{t\lambda}|{|\lambda|^{\frac{1}{q}-\frac{1}{2}-\frac{1}{p}}}|d\lambda|.
\end{equation} 
We note that if $1<q<p<2$ then $-\frac{1}{2}<\frac{1}{q}-\frac{1}{p}-\frac{1}{2}<0$, so we
can repeat the computation done in the proof of Theorem \ref{t.loc.sem.es.}. In particular
$$ I_1(t)= t^{-\frac{1}{2}-\left(\frac{1}{q}-\frac{1}{p}\right)}\int_{-\infty}^0 e^\sigma |\sigma^2+t^2\varepsilon^2|^{-\frac{1}{4}+\frac{1}{2q}-\frac{1}{2p}}d\sigma $$
$$ \leq t^{-\frac{1}{2}-\left(\frac{1}{q}-\frac{1}{p}\right)}\int_{-\infty}^0 e^\sigma |\sigma|^{-\frac{1}{2}+\frac{1}{q}-\frac{1}{p}}d\sigma\leq C t^{-\frac{1}{2}-\left(\frac{1}{q}-\frac{1}{p}\right)} $$
and the computation for $I_3$ is similar. For $I_2$, we can follow \eqref{eq.I2}.

\end{proof}

\subsection{Energy estimates}\label{subsec.en.es.}

Let $T>0$, then we want to prove the existence of a solution for the linear system
\begin{equation}\label{loc.ex.sys.lin.}
\left\{\begin{array}{ll}
     (\partial_t-\Delta_\alpha)u=f & (0,T)\times \mathbb R^2 \\
     u(0)=u_0 & \mathbb R^2,
\end{array}\right. 
\end{equation}
which is the linearized system corresponding to \eqref{m.sys.}. We can use the Duhamel Formula to write the weak definition for the solution $u$:
\begin{equation}\label{Duh.for.lin.}
    u(t)=e^{\Delta_\alpha t}u_0 + \int_0^t e^{\Delta_\alpha (t-\tau)}f(\tau)d\tau.
\end{equation}
Let $u_0\in H^1_\alpha(\mathbb R^2)$ and $f\in L^2_TL^2$. Following the estimates the authors got in \cite{BG25}, by the energy method we expect to prove the solution to satisfy 
$$ \|u\|_{L^\infty_TH^1_\alpha} + \|u\|_{L^2_TH^2_\alpha}\lesssim \|u_0\|_{H^1_\alpha} + \|f\|_{L^2_TL^2}. $$
When $T<\infty$, it is possible to prove such an estimate even for $\Delta_\alpha$. However, it is not possible to prove the local well-posedness for our problem in this way. In particular, we need an estimate with $\|f\|_{L^2_TL^h}$ and $h\in[1,2)$. As we will see, the reason comes from the decomposition of $u$ in \eqref{dom.Lap.alpha}:
$$ \nabla u=\nabla \phi_\lambda + q\nabla G_\lambda, $$
for some $\lambda\in\mathbb R$ and where $G_\lambda$ is defined in \eqref{eq.action}. In particular, $\nabla G_\lambda\in L^h(\mathbb R^2)$ if and only if $h\in[1,2)$. Since we need to estimate the gradient of our solution due to the structure of our nonlinearity in \eqref{m.sys.}, we can not expect $\nabla(|u|^\gamma)\in L^2(\mathbb R^2)$ for any $\gamma>1$. Consequently, we can not use the standard argument for the energy method: we look for a bound for the single terms of the Duhamel Formula \eqref{Duh.for.lin.}:
\begin{lem}\label{l.lin.es.loc.ex.1}
    Let $T>0$, $r>2$, $s\in\left(0,\frac{2}{r}\right)$, $p\ge 1$ and $q\in[2,\infty)$, let $u_0\in H^1_\alpha(\mathbb R^2)$, then
    $$ \left\|e^{\Delta_\alpha t} u_0\right\|_{L^\infty_T H^1_\alpha} + \left\|e^{\Delta_\alpha t} u_0\right\|_{L^r_T H^{s+1}_\alpha} + \left\|e^{\Delta_\alpha t}u_0\right\|_{L^p_TL^q}\le C(T) \|u_0\|_{H^1_\alpha(\mathbb R^2)}, $$
    for $C(T)=C(T,r,s,p)>0$ continuous for $T\ge0$.
\end{lem}
\begin{proof}\hfill\\
Choosing $\omega=1+E_\alpha$, where $E_\alpha$ is defined in \eqref{eigenvalue}, it follows from Remark \ref{rem.Lap-alpha.dom.} that
$$ \|g\|_{H^1_\alpha(\mathbb R^2)} \simeq \|(\omega-\Delta_\alpha)^{1/2} g\|_{L^2(\mathbb R^2)} . $$
Therefore, thanks to what we underlined in Remark \ref{rem.L2.sem.es.}, it holds that
$$ \left\|e^{\Delta_\alpha t} u_0\right\|_{L^\infty_T H^1_\alpha}\lesssim  \left\|(\omega - \Delta_\alpha)^{1/2}e^{\Delta_\alpha t} u_0\right\|_{L^\infty_T L^2} $$
$$ = \left\|e^{\Delta_\alpha t} (\omega - \Delta_\alpha)^{1/2}u_0\right\|_{L^\infty_T L^2} \lesssim  \left\|(\omega - \Delta_\alpha)^{1/2}u_0\right\|_{L^2(\mathbb R^2)}\lesssim \|u_0\|_{H^1_\alpha(\mathbb R^2)}. $$
Let us pass to the other estimates: 
$$ \left\|e^{\Delta_\alpha t} u_0\right\|_{H^{s+1}_\alpha(\mathbb R^2)}\lesssim \left\|(\omega - \Delta_\alpha)^{s/2}e^{\Delta_\alpha t} (\omega -\Delta_\alpha)^{1/2}u_0\right\|_{L^2(\mathbb R^2)}\lesssim \left(1+t^{-s/2}\right)\|u_0\|_{H^1_\alpha(\mathbb R^2)},  $$
where the last inequality follows from the Spectral Theorem. So, since $sr<2$, 
$$ \left\|e^{\Delta_\alpha t} u_0\right\|_{L^r_T H^{s+1}_\alpha}\lesssim \|u_0\|_{H^1_\alpha(\mathbb R^2)}. $$
Finally, the $L^pL^q$-estimate can be done directly by the semigroup estimates from Theorem \ref{t.loc.sem.es.}:
$$ \left\|e^ {\Delta_\alpha t}u_0\right\|_{L^p_TL^q}\le T^{1/p}\|u_0\|_{L^q(\mathbb R^2)}\lesssim T^{1/p}\|u_0\|_{H^1_\alpha(\mathbb R^2)},  $$
where the last inequality follows from the Sobolev embeddings of $H^1_\alpha(\mathbb R^2)$ from Remark \ref{rem.Sob.emb.}.
\end{proof}
\begin{lem}\label{l.lin.es.loc.ex.2}
    Let $T>0$, $r>2$, $s\in\left(0,\frac{2}{r}\right)$, $p\ge 1$ and $q\in[2,\infty)$ then there is $\varepsilon_0(r,s)>0$ sufficiently small such that, for any $\varepsilon<\varepsilon_0$ and for any
    $$ f\in L^r_TL^\frac{2}{1+\varepsilon}\cap L^1_TL^\frac{1}{1-\varepsilon}, $$
    the function
    $$ w(t)=\int_0^t e^{\Delta_\alpha(t-\tau)}f(\tau)d\tau $$
    satisfies
    $$ \|w\|_{L^\infty_T H^1_\alpha} + \|w\|_{L^r_T H^{s+1}_\alpha} + \|w\|_{L^p_TL^q}\le C(T)\|f\|_{L^r_TL^\frac{2}{1+\varepsilon}}, $$
    for some $C(T)=C(T,r,s,p,q,\varepsilon)>0$ continuous for $T\ge 0$.
\end{lem}
\begin{proof}\hfill\\
As in the previous proof, thanks to Theorem \ref{t.loc.sem.es.}
$$ \left\|\int_0^t e^{\Delta_\alpha(t-\tau)}f(\tau)d\tau\right\|_{H^1_\alpha(\mathbb R^2)}\lesssim \left\|(\omega-\Delta_\alpha)^{1/2}\int_0^t e^{\Delta_\alpha(t-\tau)}f(\tau)d\tau\right\|_{L^2(\mathbb R^2)} $$
$$ \lesssim \int_0^t (t-\tau)^{-\frac{1}{2}-\left(\frac{1+\varepsilon}{2}-\frac{1}{2}\right)}\|f(\tau)\|_{L^\frac{2}{1+\varepsilon}(\mathbb R^2)}d\tau $$
$$ = \int_0^t (t-\tau)^{-\frac{1+\varepsilon}{2}}\|f(\tau)\|_{L^\frac{2}{1+\varepsilon}(\mathbb R^2)}d\tau \le \left(\int_0^t (t-\tau)^{-\frac{1+\varepsilon}{2}\cdot \frac{r}{r-1}}d\tau\right)^\frac{r-1}{r}\|f\|_{L^r_TL^\frac{2}{1+\varepsilon}}. $$
In particular, since $r>2$, then $r^\prime=\frac{r}{r-1}<2$ and for $\varepsilon>0$ sufficiently small the first factor is finite. Let us pass to the $L^r_TH^{s+1}_\alpha$-norm. As before
$$ \left\|(\omega-\Delta_\alpha)^\frac{s+1}{2}\int_0^t e^{\Delta_\alpha(t-\tau)}f(\tau)d\tau\right\|_{L^r_TL^2} $$
$$ \lesssim \int_0^t (t-\tau)^{-\frac{1+s}{2}-\left(\frac{1+\varepsilon}{2}-\frac{1}{2}\right)}\|f(\tau)\|_{L^\frac{2}{1+\varepsilon}(\mathbb R^2)}d\tau =  \int_0^t (t-\tau)^{-\frac{1+s+\varepsilon}{2}}\|f(\tau)\|_{L^\frac{2}{1+\varepsilon}(\mathbb R^2)}d\tau $$ 
$$ = \int_\mathbb R (t-\tau)^{-\frac{1+s+\varepsilon}{2}}\mathbbm{1}_{(0,T)}(t-\tau)\|f(\tau)\|_{L^\frac{2}{1+\varepsilon}(\mathbb R^2)}\mathbbm{1}_{(0,T)}(\tau)d\tau $$
$$ = \left(t^{-\frac{1+s+\varepsilon}{2}}\mathbbm{1}_{(0,T)}\right)*\left(\|f\|_{L^\frac{2}{1+\varepsilon}(\mathbb R^2)}\mathbbm{1}_{(0,T)}\right). $$
So, by Young's inequality
$$ \left\|\int_0^t (t-\tau)^{-\frac{1+s+\varepsilon}{2}}\|f(\tau)\|_{L^\frac{2}{1+\varepsilon}(\mathbb R^2)}d\tau\right\|_{L^r((0,T))} $$
$$ \lesssim \left\|t^{-\frac{1+s+\varepsilon}{2}}\right\|_{L^1((0,T))}\|f\|_{L^r_TL^\frac{2}{1+\varepsilon}}\lesssim \|f\|_{L^r_TL^\frac{2}{1+\varepsilon}}, $$
for $\varepsilon>0$ sufficiently small.
Let us study now the $L^pL^q$-norm. We need to distinguish some cases:
\begin{itemize}
    \item If $p\ge r$, using Young's inequality as before
    $$ \left\|\int_0^t e^{\Delta_\alpha(t-\tau)}f(\tau)d\tau\right\|_{L^p_TL^q}\lesssim \left\|\int_0^t (t-\tau)^{-\left(\frac{1+\varepsilon}{2}-\frac{1}{q}\right)}\|f(\tau)\|_{L^\frac{2}{1+\varepsilon}(\mathbb R^2)}d\tau\right\|_{L^p((0,T))} $$
    $$ \lesssim \left\|t^{-\left(\frac{1+\varepsilon}{2}-\frac{1}{q}\right)}\right\|_{L^h((0,T))}\|f\|_{L^r_TL^\frac{2}{1+\varepsilon}},  $$
    where
    $$ 1+\frac{1}{p} = \frac{1}{h} + \frac{1}{r}. $$
    We notice that, since $p\ge r$ and $r>2$, $h\in[1,2)$. So, choosing $\varepsilon>0$ sufficiently small, it holds
     $$ \left\|\int_0^t e^{\Delta_\alpha(t-\tau)}f(\tau)d\tau\right\|_{L^p_TL^q}\lesssim \|f\|_{L^r_TL^\frac{2}{1+\varepsilon}}.  $$
    \item If $p<r$, then from the previous point we get
     $$ \left\|\int_0^t e^{\Delta_\alpha(t-\tau)}f(\tau)d\tau\right\|_{L^p_TL^q}\le T^\frac{pr}{r-p} \left\|\int_0^t e^{\Delta_\alpha(t-\tau)}f(\tau)d\tau\right\|_{L^r_TL^q}\lesssim \|f\|_{L^r_TL^\frac{2}{1+\varepsilon}}. $$
\end{itemize}
\end{proof}
As a consequence, we have the a priori estimate on the linear local problem:
\begin{prop}\label{p.lin.loc.es.}
Let $\alpha\in\mathbb R$, $T>0$, $r>2$, $s\in\left(0,\frac{2}{r}\right)$, $p\ge 1$ and $q\in[2,\infty)$, let $u_0\in H^1_\alpha(\mathbb R^2)$, then there is $\varepsilon_0(r,s)>0$ sufficiently small such that, for any $\varepsilon<\varepsilon_0$ and for any
$$ f\in L^r_TL^\frac{2}{1+\varepsilon}\cap L^1_TL^\frac{1}{1-\varepsilon}, $$
the system \eqref{loc.ex.sys.lin.} admits a unique solution 
$$ u\in L^\infty_TH^1_\alpha\cap L^rH^{s+1}_\alpha\cap L^p_TL^q, $$
with 
$$ \|u\|_{L^\infty_T H^1_\alpha} + \|u\|_{L^r_T H^{s+1}_\alpha} + \|u\|_{L^p_TL^q}\le C(T)\|f\|_{L^r_TL^\frac{2}{1+\varepsilon}}, $$
for some $C(T)=C(T,r,s,p,q,\varepsilon)>0$ continuous for $T\ge 0$.
\end{prop}
Let us pass to the case of $T=\infty$: compared with the local existence, here we have to deal with the presence of a positive eigenvalue for $\Delta_\alpha$. We can not expect the existence of a solution for the system \eqref{loc.ex.sys.lin.} for $T=+\infty$: $e^{\Delta_\alpha t}u_0$ grows exponentially in time for $u_0\in H^1_\alpha(\mathbb R^2)$. For this reason, we need to take the projection $P_{ac}$ defined in \eqref{proj.def.}: let us consider the system
\begin{equation}\label{gl.ex.lin.sys.}
    \left\{ \begin{array}{ll}
        (\partial_t-P_{ac}\Delta_\alpha)u=P_{ac} f & \mathbb R_+\times \mathbb R^2 \\
        u(0)=P_{ac}u_0 & \mathbb R^2.
    \end{array}\right. 
\end{equation}
As before, the solution of the equation \eqref{gl.ex.lin.sys.} can be written through its Duhamel Formula:
$$ u(t)= e^{P_{ac}\Delta_\alpha t}P_{ac}u_0  + \int_0^t e^{P_{ac}\Delta_\alpha(t-\tau)}P_{ac}f(\tau)d\tau. $$
\begin{lem}\label{l.lin.es.gl.ex.1}
    Let $r>2$, $s\in\left(0,\frac{2}{r}\right)$, let $p,q\in(1,\infty)$ such that $\frac{1}{p}+\frac{1}{q}<1$, let $u_0\in L^1\cap H^1_\alpha(\mathbb R^2)$, then
    $$ \left\|e^{P_{ac}\Delta_\alpha t} P_{ac} u_0\right\|_{L^\infty H^1_\alpha} + \left\|e^{P_{ac}\Delta_\alpha t} P_{ac}u_0\right\|_{L^r H^{s+1}_\alpha} +  \left\|e^{P_{ac}\Delta_\alpha t} P_{ac} u_0\right\|_{L^p L^q}\lesssim \|u_0\|_{L^1\cap H^1_\alpha(\mathbb R^2)}. $$
\end{lem}
\begin{proof}\hfill\\
Firstly, thanks to the identities
$$P_{ac}\Delta_\alpha=\Delta_\alpha P_{ac}, \quad P_{ac}e^{P_{ac}\Delta_\alpha t}u_0=e^{P_{ac}\Delta_\alpha t}P_{ac} u_0 $$
proven in Proposition \ref{p.proj.comm.} and \ref{prop.eq res=}, and the fact that 
$$ \|P_{ac} g\|_{L^h(\mathbb R^2)}\lesssim \|g\|_{L^h(\mathbb R^2)} \quad g\in L^h\left(\mathbb R^2\right),\quad h\in(1,\infty),  $$
we can suppose for simplicity $u_0=P_{ac} u_0$. For this reason, we can write
$$ \|u_0\|_{H^1_\alpha(\mathbb R^2)} \simeq \|(1-P_{ac}\Delta_\alpha)^{1/2} u_0\|_{L^2(\mathbb R^2)} . $$ 
Therefore,
$$ \left\|e^{P_{ac}\Delta_\alpha t} u_0\right\|_{L^\infty H^1_\alpha}\lesssim  \left\|(1-P_{ac}\Delta_\alpha)^{1/2}e^{P_{ac}\Delta_\alpha t} u_0\right\|_{L^\infty L^2} $$
$$ = \left\|e^{P_{ac}\Delta_\alpha t} (1-P_{ac}\Delta_\alpha)^{1/2}u_0\right\|_{L^\infty L^2} \lesssim  \left\|(1-P_{ac}\Delta_\alpha)^{1/2}u_0\right\|_{L^2(\mathbb R^2)}\lesssim \|u_0\|_{H^1_\alpha(\mathbb R^2)}. $$
For the other two estimates we consider only the case $t\ge 1$, since for $t<1$ we can repeat the argument of Lemma \ref{l.lin.es.loc.ex.1}. By the Spectral Theory, it holds
$$ \|g\|_{H^{s+1}_\alpha(\mathbb R^2)}\lesssim \|g\|_{L^2(\mathbb R^2)} + \|(-P_{ac}\Delta_\alpha)^{\frac{s+1}{2}}g\|_{L^2(\mathbb R^2)} $$
for any $g=P_{ac} g \in H^{s+1}_\alpha(\mathbb R^2)$. Moreover, 
$$ \left\|(-P_{ac}\Delta_\alpha)^{\frac{s+1}{2}}e^{P_{ac}\Delta_\alpha t} u_0\right\|_{L^2(\mathbb R^2)}\lesssim t^{-\frac{s+1}{2}}\|e^{P_{ac}\Delta_\alpha t/2}u_0\|_{L^2(\mathbb R^2)}\lesssim \|e^{P_{ac}\Delta_\alpha t/2}u_0\|_{L^2(\mathbb R^2)}. $$
So, we focus on the $L^2(\mathbb R^2)$-term: by Theorem \ref{t.expDelta alpha}
$$ \left\|e^{P_{ac}\Delta_\alpha t} u_0\right\|_{L^2(\mathbb R^2)}\lesssim t^{-\left(1-\varepsilon - \frac{1}{2}\right)}\|u_0\|_{L^\frac{1}{1-\varepsilon}(\mathbb R^2)}\lesssim t^{-\frac{1}{2}+\varepsilon}\|u_0\|_{L^1\cap L^2(\mathbb R^2)},  $$
for any $\varepsilon\in\left(0,\frac{1}{2}\right]$. In particular, since $r>2$, we can find $\varepsilon>0$ sufficiently small such that
$$ \left\|e^{P_{ac}\Delta_\alpha t} u_0\right\|_{L^r((1,\infty);L^2(\mathbb R^2))}\lesssim \|u_0\|_{L^1\cap H^1_\alpha(\mathbb R^2)}. $$
Finally, let us pass to the $L^pL^q$-norm: 
$$ \left\|e^{P_{ac}\Delta_\alpha t}u_0\right\|_{L^q(\mathbb R^2)}\lesssim t^{-\left(1-\varepsilon-\frac{1}{q}\right)}\|u_0\|_{L^\frac{1}{1-\varepsilon}(\mathbb R^2)}. $$
So, since 
$$ p\left(1-\frac{1}{q}\right)>1 \quad \forall q\in(1,\infty), $$
for any $\varepsilon>0$ sufficiently small it holds 
$$ \left\|e^{P_{ac}\Delta_\alpha t}u_0\right\|_{L^p((1,\infty);L^q(\mathbb R^2))}\lesssim \|u_0\|_{L^1\cap L^2(\mathbb R^2)}. $$
\end{proof}
\begin{lem}\label{l.lin.es.gl.ex.2}
Let $r>2$, $s\in\left(0,\frac{2}{r}\right)$, let $p,q\in(1,\infty)$ such that $\frac{1}{p}+\frac{1}{q}<1$, let 
$$ f\in L^rL^\frac{2}{1+\varepsilon}\cap L^1L^\frac{1}{1-\varepsilon} $$
for $\varepsilon>0$ sufficiently small, let 
$$ w(t)=\int_0^t e^{P_{ac}\Delta_\alpha(t-\tau)}f(\tau)d\tau, $$
then
$$ \|w\|_{L^\infty H^1_\alpha} + \|w\|_{L^r H^{s+1}_\alpha} + \|w\|_{L^pL^q}\lesssim \|f\|_{L^rL^\frac{2}{1+\varepsilon}} + \|f\|_{L^1L^\frac{1}{1-\varepsilon}}. $$
\end{lem}
\begin{proof}\hfill\\
As before we can suppose $f=P_{ac} f$ and $t\ge 1$, so that 
$$ \left\|\int_0^t e^{P_{ac}\Delta_\alpha(t-\tau)}f(\tau)d\tau\right\|_{H^1_\alpha(\mathbb R^2)} $$
$$ \le \left\|\int_0^{t-1} e^{P_{ac}\Delta_\alpha(t-\tau)}f(\tau)d\tau\right\|_{H^1_\alpha(\mathbb R^2)} + \left\|\int_{t-1}^t e^{P_{ac}\Delta_\alpha(t-\tau)}f(\tau)d\tau\right\|_{H^1_\alpha(\mathbb R^2)}. $$
The second term, since $t-\tau\in(0,1)$, can be done as for the local case, so we focus only on the first term: by Spectral Theory
$$ \left\|\int_0^{t-1} e^{P_{ac}\Delta_\alpha(t-\tau)}f(\tau)d\tau\right\|_{H^1_\alpha(\mathbb R^2)} $$
$$ \lesssim \left\|\int_0^t e^{P_{ac}\Delta_\alpha(t-\tau)}f(\tau)d\tau\right\|_{L^2(\mathbb R^2)} + \left\|\int_0^t (-P_{ac}\Delta_\alpha)^{1/2}e^{P_{ac}\Delta_\alpha(t-\tau)}f(\tau)d\tau\right\|_{L^2(\mathbb R^2)} $$
$$ \lesssim \int_0^{t-1} (t-\tau)^{-1+\varepsilon+\frac{1}{2}}\|f(\tau)\|_{L^\frac{1}{1-\varepsilon}(\mathbb R^2)}d\tau + \int_0^{t-1} (t-\tau)^{-1+\varepsilon+1}\|f(\tau)\|_{L^\frac{1}{1-\varepsilon}(\mathbb R^2)}d\tau \le \|f\|_{L^1L^\frac{1}{1-\varepsilon}},  $$
where we used that $t-\tau\ge 1$ for $\tau\le t-1$. Let us pass to the $L^rH^{s+1}_\alpha$-norm. As before, it is sufficient to estimate 
$$ \left\|\int_0^{t-1} e^{P_{ac}\Delta_\alpha(t-\tau)}f(\tau)d\tau\right\|_{L^r((1,\infty);H^{s+1}_\alpha(\mathbb R^2))}. $$
Again, by Spectral Theory and the semigroup estimates from Theorem \ref{t.expDelta alpha}, since $t-\tau\ge 1$, it holds
$$ \left\|\int_0^{t-1} e^{P_{ac}\Delta_\alpha(t-\tau)}f(\tau)d\tau\right\|_{L^r((1,\infty);H^{s+1}_\alpha(\mathbb R^2))} $$
$$ \lesssim \left\|\int_0^{t-1} (t-\tau)^{-\left(1-\varepsilon-\frac{1}{2}\right)}\|f(\tau)\|_{L^\frac{1}{1-\varepsilon}(\mathbb R^2)}d\tau\right\|_{L^r((1,\infty))} $$
$$ \lesssim \left\|t^{-\frac{1-2\varepsilon}{2}}\right\|_{L^r((1,\infty))}\|f\|_{L^1L^\frac{1}{1-\varepsilon}} \lesssim \|f\|_{L^1L^\frac{1}{1-\varepsilon}}, $$
for $\varepsilon>0$ sufficiently small, where we used Young's inequality and the condition $r>2$. Finally,
$$ \left\|\int_0^{t-1} e^{P_{ac}\Delta_\alpha(t-\tau)}f(\tau)d\tau\right\|_{L^p((1,\infty);L^q(\mathbb R^2))}\lesssim \left\|\int_0^{t-1} (t-\tau)^{-\left(1-\varepsilon - \frac{1}{q}\right)}\|f(\tau)\|_{L^\frac{1}{1-\varepsilon}(\mathbb R^2)}d\tau\right\|_{L^p((1,\infty))} $$
$$ \lesssim  \left\|t^{-\left(1-\varepsilon-\frac{1}{q}\right)}\right\|_{L^p((1,\infty))}\|f\|_{L^1L^\frac{1}{1-\varepsilon}}\lesssim \|f\|_{L^1L^\frac{1}{1-\varepsilon}}, $$
since
$$ p\left(1-\varepsilon-\frac{1}{q}\right)<1 $$
for $p>\frac{q}{q-1}$ and $\varepsilon>0$ sufficiently small.
\end{proof}
So we get the existence of a solution for \eqref{gl.ex.lin.sys.}:
\begin{prop}\label{p.gl.lin.es.}
Let $\alpha\in\mathbb R$, $r>2$, $s\in\left(0,\frac{2}{r}\right)$, let $p,q\in(1,\infty)$ such that $\frac{1}{p}+\frac{1}{q}<1$, let $u_0\in L^1\cap H^1_\alpha(\mathbb R^2)$, then there is $\varepsilon>0$ sufficiently small such that, for
$$ f\in L^rL^\frac{2}{1+\varepsilon}\cap L^1L^\frac{1}{1-\varepsilon}, $$
the system \eqref{gl.ex.lin.sys.} admits a unique solution $u$ such that
$$ \|u\|_{L^\infty H^1_\alpha} + \|u\|_{L^r H^{s+1}_\alpha} + \|u\|_{L^pL^q}\lesssim \|u_0\|_{L^1\cap H^1_\alpha(\mathbb R^2)} + \|f\|_{L^rL^\frac{2}{1+\varepsilon}} + \|f\|_{L^1L^\frac{1}{1-\varepsilon}}. $$
\end{prop}
We can say something more about the system \eqref{gl.ex.lin.sys.}: the presence of $P_{ac}f$ and of $P_{ac}u_0$ forces the solution to satisfy the condition $P_{ac}u=u$: due to the identity 
\begin{equation}\label{proj.id.2}
P_{ac}e^{P_{ac}\Delta_\alpha t} g= e^{P_{ac}\Delta_\alpha t}P_{ac}g \quad  g\in L^h\left(\mathbb R^2\right),\quad h\in(1,\infty),
\end{equation}
we notice that
$$ u(t)=e^{P_{ac}\Delta_\alpha t}P_{ac}u_0 + \int_0^t e^{P_{ac}\Delta_\alpha(t-\tau)}P_{ac}f(\tau)d\tau = P_{ac}\left[e^{P_{ac}\Delta_\alpha t}u_0 + \int_0^te^{P_{ac}\Delta_\alpha (t-\tau)}f(\tau)d\tau\right].  $$
In particular, since $P_{ac}\Delta_\alpha =\Delta_\alpha P_{ac}$, it holds
$$ (\partial_t-P_{ac}\Delta_\alpha)u=(\partial_t-\Delta_\alpha)u. $$
So, $u$ solves
\begin{equation}\label{sys.aux.}
    \left\{\begin{array}{ll}
    (\partial_t-\Delta_\alpha)u=P_{ac}f & \mathbb R_+\times\mathbb R^2 \\
    P_{ac}u=u & \mathbb R_+\times \mathbb R^2 \\
    u(0)=P_{ac}u_0 & \mathbb R^2.
\end{array}\right.
\end{equation}
In general, we are not allowed to delete the projection on $f$ in the right hand side of the equation \eqref{sys.aux.}. However, it is possible if we introduce a Lagrange multiplier associated with the condition $P_{ac}u=u$:
\begin{prop}\label{p.ex.Lagr.sys.}
Let $\alpha\in\mathbb R$, $r>2$, $s\in\left(0,\frac{2}{r}\right)$, let $p,q\in(1,\infty)$ such that $\frac{1}{p}+\frac{1}{q}<1$, let $u_0$ and $f$ as in Proposition \ref{p.gl.lin.es.}, then there is a unique couple $(u,\rho)$ which solves the system
\begin{equation}\label{lin.gl.sys.Lagr.}
\left\{\begin{array}{ll}
    (\partial_t-\Delta_\alpha)u + \psi_\alpha \rho=f  & \mathbb R_+\times \mathbb R^2 \\
    P_{ac}u=u & \mathbb R_+\times \mathbb R^2 \\
    u(0)=P_{ac}u_0 & \mathbb R^2,
\end{array}\right. 
\end{equation}
such that 
$$ \|u\|_{L^\infty H^1_\alpha} + \|u\|_{L^r H^{s+1}_\alpha} + \|u\|_{L^pL^q} + \|\rho\|_{L^r\cap L^1(\mathbb R_+)}\lesssim \|u_0\|_{L^1\cap H^1_\alpha(\mathbb R^2)} + \|f\|_{L^rL^\frac{2}{1+\varepsilon}} + \|f\|_{L^1L^\frac{1}{1-\varepsilon}}, $$
where $\psi_\alpha$ was defined in \eqref{eigenfuntion}. Moreover, let
$$ u\in L^\infty H^1_\alpha\cap L^r H^{s+1}_\alpha \cap L^pL^q, $$
then $u$ solves the system \eqref{gl.ex.lin.sys.} if and only if $(u,\rho)$ solves \eqref{lin.gl.sys.Lagr.} with 
$$ \rho(t)=\left<f(t),\psi_\alpha\right>_{L^2(\mathbb R^2)}\quad t>0. $$
\end{prop}
\begin{proof}\hfill\\
Let $u$ be the solution from Proposition \ref{p.gl.lin.es.}. We have already shown that $u$ solves the system \eqref{sys.aux.}, that is
$$ (\partial_t-\Delta_\alpha)u=P_{ac}f=f-P_df. $$
We recall that 
$$ P_df=\left<f,\psi_\alpha\right>\psi_\alpha. $$
Therefore it is sufficient to notice that 
$$ |\left<f,\psi_\alpha\right>|\lesssim \min\left\{\|f\|_{L^\frac{2}{1+\varepsilon}};\|f\|_{L^\frac{1}{1-\varepsilon}}\right\}, $$
where we used that $\psi_\alpha\in L^h(\mathbb R^2)$ for any $h\in[1,\infty)$. Let us prove that, if $(u,\rho)$ is a solution of \eqref{lin.gl.sys.Lagr.}, then $u$ solves \eqref{gl.ex.lin.sys.}: $P_{ac}u=u$, so 
$$ u(t)=P_{ac}u(t)=P_{ac}e^{\Delta_\alpha t}u_0 + P_{ac}\int_0^te^{\Delta_\alpha(t-\tau)}(f + \psi_\alpha \rho)(\tau)d\tau $$
$$ = e^{P_{ac}\Delta_\alpha t}P_{ac}u_0 +  \int_0^te^{P_{ac}\Delta_\alpha(t-\tau)}P_{ac}f(\tau)d\tau, $$
where we used the identity \eqref{proj.id.2}. In particular $u$ solves \eqref{gl.ex.lin.sys.}. This also proved the uniqueness of the solutions of \eqref{lin.gl.sys.Lagr.}: let $(u_1,\rho_1)$ and $(u_2,\rho_2)$ be two solutions of \eqref{lin.gl.sys.Lagr.}, then we have just proved that $u_1=u_2$ is the only solution of \eqref{gl.ex.lin.sys.}. Finally
$$ f=(\partial_t-\Delta_\alpha)u_j  + \rho_j\psi_\alpha=(\partial_t-P_{ac}\Delta_\alpha)u_j + \rho_j\psi_\alpha=P_{ac}f + \rho_j\psi_\alpha. $$
Therefore 
$$ \rho_j\psi_\alpha=P_df \quad j=1,2. $$
\end{proof}

\section{Proof of Theorem \ref{t.ex.loc.main}}

Let us consider the system 
\begin{equation}\label{loc.ex.sys.}
\left\{\begin{array}{ll}
     (\partial_t-\Delta_\alpha)u= a\cdot \nabla(|u|^\gamma) & (0,T)\times \mathbb R^2 \\
     u(0)=u_0 & \mathbb R^2.
\end{array}\right. 
\end{equation}
To prove the existence of a solution for \eqref{loc.ex.sys.}, we want to solve the Duhamel equation: 
$$ u(t)=e^{\Delta_\alpha t}u_0 + \int_0^t e^{\Delta_\alpha(t-\tau)}a\cdot\nabla(|u|^\gamma)(\tau)d\tau. $$
We consider the map 
$$ \Phi(v)=e^{\Delta_\alpha t}u_0 + \int_0^t e^{\Delta_\alpha(t-\tau)}a\cdot\nabla(|v|^\gamma)(\tau)d\tau. $$
Thanks to the a priori estimate we found in Subsection \ref{subsec.en.es.}, we are going to prove that $\Phi\colon V\to V$ is a contraction, for a proper choice of $V\subseteq L^2(\mathbb R^2)$ Banach space. 
\begin{proof}[Proof of Theorem \ref{t.ex.loc.main}]\hfill\\
Let us fix $s<\frac{2}{r}$, $p\ge 1$, $q\in[2,\infty)$ and let us define the spaces
$$ X_T= L^\infty_T H^1_\alpha\cap L^r_T H^{s+1}_\alpha\cap L^p_TL^q $$
and
$$ Y_T=\left\{v\in X_T\mid v(0)=v_0\right\}, $$
endowed with the norm 
$$ \|v\|_{Y_T}=\|v\|_{L^\infty_TH^s_\alpha} + \|v\|_{L^2_TH^{s+1}_\alpha}+\|v\|_{L^p_TL^q}. $$
Let $\Phi$ be the map such that $\Phi(v)=u$ for $v\in Y_T$, where $u$ solves
$$ \left\{\begin{array}{ll}
    (\partial_t-\Delta_\alpha)u=a\cdot\nabla (|v|^\gamma) & (0,T)\times\mathbb R^2 \\
    u(0)=u_0 & \mathbb R^2. 
\end{array}\right. $$
Firstly, we want to prove that $\Phi\colon Y_T\to Y_T$. Since $v(t)\in H^{s+1}_\alpha(\mathbb R^2)$ for a.e. $t\in(0,T)$, we have from Remark \ref{rem.Lap-alpha.dom.} that 
$$ v(t,x)=\phi_\lambda(t,x) + q(t)G_\lambda(x) \quad \text{for a.e.}\quad  (t,x)\in(0,T)\times\mathbb R^2, $$
with $\phi_\lambda(t)\in H^{s+1}(\mathbb R^2)$ for a.e. $t\in(0,T)$. Moreover
$$ |\nabla(|v|^\gamma)|\lesssim |\nabla v||v|^{\gamma-1}\le |\nabla \phi_\lambda||v|^{\gamma-1} + |q||\nabla G_\lambda||v|^{\gamma-1}. $$
We want to use the linear estimates from Lemma \ref{l.lin.es.loc.ex.1} and \ref{l.lin.es.loc.ex.2}. So, we need to prove that 
$$  \||\nabla \phi_\lambda||v|^{\gamma-1}\|_{L^r_TL^\frac{2}{1+\varepsilon}} + \||q||\nabla G_\lambda||v|^{\gamma-1}\|_{L^r_TL^\frac{2}{1+\varepsilon}}<+\infty, $$
for some small $\varepsilon>0$. 
$$ \||\nabla \phi_\lambda| |v|^{\gamma-1}\|_{L^r_TL^\frac{2}{1+\varepsilon}}\le \left\|\|\nabla \phi_\lambda\|_{L^2(\mathbb R^2)} \|v\|_{L^\frac{2(\gamma-1)}{\varepsilon}(\mathbb R^2)}\right\|_{L^r((0,T))}. $$
Due to Remark \ref{rem.Sob.emb.}, $H^1_\alpha(\mathbb R^2)\hookrightarrow L^h(\mathbb R^2)$ for any $h\in[2,\infty)$, so
$$ \left\|\|\nabla \phi_\lambda\|_{L^2(\mathbb R^2)} \|v\|_{L^\frac{2(\gamma-1)}{\varepsilon}(\mathbb R^2)}\right\|_{L^r((0,T))}\lesssim \|v\|_{L^\infty_T H^1_\alpha}^{\gamma-1}\|\nabla \phi_\lambda\|_{L^r_TL^2} $$
$$ \le T^{1/r}\|v\|_{L^\infty_T H^1_\alpha}^\gamma\lesssim T^{1/r}\|v\|_Y^\gamma. $$
On the other hand,
$$ \||q||\nabla G_\lambda| |v|^{\gamma-1}\|_{L^r_TL^\frac{2}{1+\varepsilon}}\le \left\||q|\|\nabla G_\lambda\|_{L^\frac{2}{1+\frac{1}{2}\varepsilon}(\mathbb R^2)} \|v\|_{L^\frac{4(\gamma-1)}{\varepsilon}(\mathbb R^2)}^{\gamma-1}\right\|_{L^r((0,T))} $$
$$ \lesssim \|v\|_{L^\infty_T H^1_\alpha}^{\gamma-1}\|v\|_{L^r_TH^{1+\delta}_\alpha}, $$
for any $\delta>0$. In particular, choosing $\delta>0$ small, by interpolation it holds
$$ \|v\|_{H^{1+\delta}(\mathbb R^2)}\lesssim \|v\|_{H^1(\mathbb R^2)}^\theta\|v\|_{H^{s+1}(\mathbb R^2)}^{1-\theta}, $$
with $\theta=1-\frac{\delta}{s}$. Then
$$ \|v\|_{L^\infty_T H^1_\alpha}^{\gamma-1}\|v\|_{L^r_TH^{1+\delta}_\alpha}\lesssim T^\beta \|v\|_Y^\gamma, $$
for some $\beta>0$. So
$$ \|\nabla(|v|^\gamma)\|_{L^rL^\frac{2}{1+\varepsilon}}\lesssim T^\beta\|v\|_Y^\gamma, $$
and from Proposition \ref{p.lin.loc.es.}, we have that 
\begin{equation}\label{proof.loc.es.1}
    \|\Phi(v)\|_{Y}\le CT^\beta \left[\|u_0\|_{H^1_\alpha(\mathbb R^2)} + \|v\|^\gamma_Y\right]
\end{equation}
for some $C,\beta >0$. Moreover
$$ \left|\nabla(|v_1|^\gamma)-\nabla(|v_2|^\gamma)\right|=\gamma\left|v_1|v_1|^{\gamma-2}\nabla v_1-v_2|v_2|^{\gamma-2}\nabla v_2\right| $$
$$ \le \gamma\left[|v_1-v_2||v_1|^{\gamma-2}|\nabla v_1| + (\gamma-2)|v_2|^{\gamma-2}|v_1-v_2||\nabla v_1| + |v_2|^{\gamma-1}|\nabla(v_1-v_2)|\right]. $$
So, similarly, we get
\begin{equation}\label{proof.loc.es.2}
    \|\Phi(v_1)-\Phi(v_2)\|_{Y_T}\le  MT^\beta \left(\|v_1\|_{Y_T}^{\gamma-1} + \|v_2\|_{Y_T}^{\gamma-1}\right)\|v_1-v_2\|_{Y_T},
\end{equation}
for some $M,\beta>0$. Finally, let us define the space
$$ Z_\omega=\{v\in Y_T\mid \|v\|_{Y_T}\le \omega\}. $$
We want to prove that $\Phi\colon Z_\omega\to Z_\omega$ for a suitable $\omega, T>0$. This comes from \eqref{proof.loc.es.1}:
$$ \|\Phi(v)\|_{Y_T}\le C\left[\|u_0\|_{H^1_\alpha(\mathbb R^2)} + T^\beta \omega^\gamma\right]. $$
So, choosing 
$$ \omega = 2C\|u_0\|_{H^1_\alpha(\mathbb R^2)} $$
and $T>0$ sufficiently small such that
$$ T^\beta\omega^\gamma<\|u_0\|_{H^1_\alpha(\mathbb R^2)}, $$
we get that $\Phi\colon Z_\omega \to Z_\omega$. Moreover, from \eqref{proof.loc.es.2}
$$ \|\Phi(v_1)-\Phi(v_2)\|_{Y_T}\le  2\omega^{\gamma-1} MT^\beta\|v_1-v_2\|_{Y_T}. $$
So, for $T>0$ sufficiently small, we get that $\Phi\colon Z_\omega\to Z_\omega$ is a contraction. Finally, by the Fixed Point Theorem, there is a unique solution $u\in Z_\omega$ for the system \eqref{loc.ex.sys.}. 

Let us prove now that $u$ is unique in $Y_T$: let $u_1, u_2\in Y_T$ and let us denote 
$$ R_j=\|u_j\|_{Y_T}\quad j=1,2 . $$
In particular, $u_1$ and $u_2$ are fixed points for the function $\Phi$. Then, for any $T_0<T$, by the estimate \eqref{proof.loc.es.2} we get
$$ \|u_1-u_2\|_{Y_{T_0}}\lesssim T_0^\beta \left(R_1^{\gamma-1}+R_2^{\gamma-2}\right)\|u_1-u_2\|_{Y_{T_0}}. $$
So, if we choose $T_0=T_0(R_1,R_2)$ sufficiently small, we get that $u_1(t)=u_2(t)$ for a.e. $t\in(0,T_0)$. Since the choice of $T_0$ does not depend on $u_0$, we can repeat the argument choosing as starting point $T_0$. In this way, in a finite number of steps, we prove that $u_1=u_2$ in $Y_T$.

Finally, to conclude, it is sufficient to notice that the estimate \eqref{proof.loc.es.1} works for any $s\in\left(0,\frac{2}{r}\right)$ and for any $p\ge1$ and $q\in[2,\infty)$, so we can prove analogously that 
$$ u\in L^r_TH^{s+1}_\alpha\cap L^p_TL^q \quad \forall s\in\left(0,\frac{2}{r}\right),\quad p\ge 1,\quad q\in[2,\infty). $$
\end{proof}

\section{Global Existence}
\subsection{Proof of Theorems \ref{t.ex.gl.m.} and \ref{t.ex.gl.m.2}}

Let us consider the systems
\begin{equation}\label{sys.gl.}
    \left\{\begin{array}{ll}
        (\partial_t-P_{ac}\Delta_\alpha)u=P_{ac}(a\cdot\nabla)(|u|^\gamma) & \mathbb R_+\times\mathbb R^2 \\
        P_{ac}u=u & \mathbb R_+\times\mathbb R^2 \\
        u_0=P_{ac}u_0 & \mathbb R^2
    \end{array}\right.
\end{equation}
and 
\begin{equation}\label{sys.gl.2}
    \left\{\begin{array}{ll}
        (\partial_t-\Delta_\alpha)u + \rho\psi_\alpha = a\cdot\nabla(|u|^\gamma) & \mathbb R_+\times\mathbb R^2 \\
        P_{ac}u=u & \mathbb R_+\times\mathbb R^2 \\
        u_0=P_{ac}u_0 & \mathbb R^2,
    \end{array}\right.
\end{equation}
where $\psi_\alpha$ is the eigenfunction \eqref{eigenfuntion} for $\Delta_\alpha$. The strategy for the system \eqref{sys.gl.} is, as before, to apply a contraction argument on the map
$$ \Phi(v)=e^{P_{ac}\Delta_\alpha t}P_{ac}u_0 + \int_0^t e^{P_{ac}\Delta_\alpha(t-\tau)}P_{ac}(a\cdot \nabla)(|v|^\gamma)d\tau. $$
Once we prove the existence for the system \eqref{sys.gl.}, thanks to Proposition \ref{p.gl.lin.es.}, we have that 
$$ \rho(t)=\left<a\cdot \nabla(|u|^\gamma),\psi_\alpha\right>_{L^2(\mathbb R^2)} $$
and we get the existence for the system \eqref{sys.gl.2}.
\begin{proof}[Proof of Theorems \ref{t.ex.gl.m.} and \ref{t.ex.gl.m.2}]\hfill\\
Let 
$$ X= L^\infty H^1_\alpha\cap L^r H^{s+1}_\alpha L^{p_1}L^{q_1}\cap L^{p_2}L^{q_2}, $$
for $p_1,p_2\in (1,\infty)$ and $q_1,q_2\in[2,\infty)$ with
$$ \frac{1}{p_i} + \frac{1}{q_i} < 1\quad i=1,2 $$
that will be choose properly later. Moreover, let us define 
$$ Y=\left\{v\in X\mid v(0)=v_0\right\}, $$
endowed with the norm 
$$ \|v\|_{Y}=\|v\|_{L^\infty_TH^s_\alpha} + \|v\|_{L^2_TH^{s+1}_\alpha} + \|v\|_{L^{p_1}L^{q_1}} + \|v\|_{L^{p_2}L^{q_2}}. $$
We consider now the map $\Phi(v)=u$ for $v\in Y$, where $u$ solves
$$ \left\{\begin{array}{ll}
    (\partial_t-P_{ac}\Delta_\alpha)u=P_{ac}(a\cdot \nabla) (|v|^\gamma) & \mathbb R_+\times\mathbb R^2 \\
    u(0)=P_{ac}u_0 & \mathbb R^2. 
\end{array}\right. $$
As before, let us start proving that $\psi\colon Y\to Y$. From Remark \ref{rem.Lap-alpha.dom.} 
$$ v(t,x)=\phi_\lambda(t,x) + q(t)G_\lambda(x) \quad \text{for a.e.}\quad  (t,x)\in\mathbb R_+\times\mathbb R^2, $$
with $\phi_\lambda(t)\in H^{s+1}(\mathbb R^2)$ for a.e. $t>0$. 
$$ |\nabla(|v|^\gamma)|\lesssim |\nabla v||v|^{\gamma-1}\le |\nabla \phi_\lambda||v|^{\gamma-1} + |q||\nabla G_\lambda||v|^{\gamma-1}. $$
We want to use the linear estimates from Lemma \ref{l.lin.es.gl.ex.1} and \ref{l.lin.es.gl.ex.2} and therefore we need to control
$$ \||\nabla \phi_\lambda||v|^{\gamma-1}\|_{L^1L^\frac{1}{1-\varepsilon}} + \||\nabla \phi_\lambda||v|^{\gamma-1}\|_{L^rL^\frac{2}{1+\varepsilon}} $$
and 
$$ \||q||\nabla G_\lambda||v|^{\gamma-1}\|_{L^1L^\frac{1}{1-\varepsilon}} + \||q||\nabla G_\lambda||v|^{\gamma-1}\|_{L^rL^\frac{2}{1+\varepsilon}}, $$
for $\delta>0$ sufficiently small.
$$ \||\nabla \phi_\lambda||v|^{\gamma-1}\|_{L^1L^\frac{1}{1-\varepsilon}(\mathbb R^2)}\le \|\nabla \phi_\lambda\|_{L^r L^2}\|v\|_{L^\frac{r(\gamma-1)}{r-1}L^\frac{2(\gamma-1)}{1-2\varepsilon}}^{\gamma-1}. $$
We notice that 
$$ \|\nabla \phi_\lambda\|_{L^rL^2}\lesssim \|v\|_{L^rH^{s+1}_\alpha}. $$
Moreover, 
$$ \frac{r(\gamma-1)}{r-1}>\frac{2(\gamma-1)}{2(\gamma-1)-1+2\varepsilon}=\frac{2(\gamma-1)}{2\gamma-3+2\varepsilon}\:\Leftrightarrow\: r(2\gamma-3+2\varepsilon)>2(r-1). $$
If $\gamma \ge \frac{5}{2}$ the estimate is true. Otherwise we need to take $r$ such that
$$ r<\frac{2}{5-2\gamma-2\varepsilon}. $$
This is possible if and only if
$$ 2<\frac{2}{5-2\gamma -2\varepsilon}\:\Leftrightarrow\: 5-2\gamma-2\varepsilon<1\:\Leftrightarrow\: \gamma>2-\varepsilon. $$
In particular is true for $\gamma\ge 2$. So, if we choose 
$$ (p_1,q_1)=\left(\frac{r(\gamma-1)}{r-1},\frac{2(\gamma-1)}{1-2\varepsilon}\right), $$
we get
$$ \||\nabla \phi_\lambda||v|^{\gamma-1}\|_{L^1L^\frac{1}{1-\varepsilon}(\mathbb R^2)}\le \|\nabla \phi_\lambda\|_{L^r L^2}\|v\|_{L^\frac{r(\gamma-1)}{r-1}L^\frac{2(\gamma-1)}{1-2\varepsilon}}^{\gamma-1}\lesssim \|v\|_Y^\gamma. $$
Similarly, using the fact that $\nabla G_\lambda\in L^h(\mathbb R^2)$ for $h\in[1,2)$ and
$$ |q|\lesssim \|\phi_\lambda\|_{H^{s+1}_\alpha(\mathbb R^2)}\lesssim \|v\|_{H^{s+1}_\alpha(\mathbb R^2)},  $$
we get 
$$ \||q||\nabla G_\lambda||v|^{\gamma-1}\|_{L^1L^\frac{1}{1-\varepsilon}}\le \left\||q|\|\nabla G_\lambda\|_{L^\frac{2}{1+\varepsilon}(\mathbb R^2)}\|v\|_{L^\frac{2(\gamma-1)}{1-3\varepsilon}(\mathbb R^2)}^{\gamma-1}\right\|_{L^1(\mathbb R_+)} $$
$$ \lesssim \left\|\|v\|_{H^{s+1}_\alpha(\mathbb R^2)}\|v\|_{L^\frac{2(\gamma-1)}{1-3\varepsilon}(\mathbb R^2)}^{\gamma -1}\right\|_{L^1(\mathbb R_+)}\le \|v\|_{L^rH^{s+1}_\alpha}\|v\|_{L^\frac{r(\gamma-1)}{r-1}L^\frac{2(\gamma-1)}{1-3\varepsilon}}\le \|v\|_{Y}^\gamma. $$
Therefore, choosing
$$ (p_2,q_2)=\left(\frac{r(\gamma-1)}{r-1},\frac{2(\gamma-1)}{1-3\varepsilon}\right), $$
we have 
\begin{equation}\label{gl.es.1}
    \|\nabla(|v|^\gamma)\|_{L^1L^\frac{1}{1-\varepsilon}}\lesssim \|v\|_Y^\gamma.
\end{equation}
Let us consider now the $L^rL^\frac{2}{1+\varepsilon}$-norm:
$$ \||\nabla \phi_\lambda| |v|^{\gamma-1}\|_{L^rL^\frac{2}{1+\varepsilon}}\le \left\|\|\nabla \phi_\lambda\|_{L^2(\mathbb R^2)} \|v\|_{L^\frac{2(\gamma-1)}{\varepsilon}(\mathbb R^2)}^{\gamma-1}\right\|_{L^r(\mathbb R_+)}. $$
Thanks to the Sobolev embedding $H^1_\alpha(\mathbb R^2)\hookrightarrow L^h(\mathbb R^2)$ from Remark \ref{rem.Sob.emb.}, it holds
$$ \left\|\|\nabla \phi_\lambda\|_{L^2(\mathbb R^2)} \|v\|_{L^\frac{2(\gamma-1)}{\varepsilon}(\mathbb R^2)}^{\gamma-1}\right\|_{L^r(\mathbb R_+)}\lesssim \|v\|_{L^\infty H^1_\alpha}^{\gamma-1}\|\nabla \phi_\lambda\|_{L^rL^2}\lesssim \|v\|_{L^\infty H^1_\alpha}^{\gamma-1}\|v\|_{L^rH^{s+1}_\alpha}\lesssim \|v\|_Y^\gamma. $$
On the other hand
$$ \left\||q||\nabla G_\lambda| |v|^{\gamma-1}\right\|_{L^rL^\frac{2}{1+\varepsilon}}\le \left\||q|\|\nabla G_\lambda\|_{L^\frac{2}{1+\frac{1}{2}\varepsilon}(\mathbb R^2)} \|v\|_{L^\frac{4(\gamma-1)}{\varepsilon}(\mathbb R^2)}^{\gamma-1}\right\|_{L^r(\mathbb R_+)} $$
$$ \lesssim \|v\|_{L^\infty H^1_\alpha}^{\gamma-1}\|v\|_{L^rH^{s+1}_\alpha} \le \|v\|_Y^\gamma. $$
So
\begin{equation}\label{gl.es.2}
    \|\nabla(|v|^\gamma)\|_{L^rL^\frac{2}{1+\varepsilon}}\lesssim \|v\|_Y^\gamma.
\end{equation}
Thanks to \eqref{gl.es.1}, \eqref{gl.es.2} and Proposition \ref{p.gl.lin.es.}, we have that 
\begin{equation}\label{proof.gl.es.3}
    \|\Phi(v)\|_{Y}\le C\left[\|u_0\|_{L^1\cap H^1_\alpha(\mathbb R^2)} + \|v\|^\gamma_Y\right],
\end{equation}
for some $C>0$, and similarly
\begin{equation}\label{proof.gl.es.4}
    \|\Phi(v_1)-\Phi(v_2)\|_{Y}\le  M\left(\|v_1\|_{Y}^{\gamma-1} + \|v_2\|_{Y}^{\gamma-1}\right)\|v_1-v_2\|_{Y},
\end{equation}
for some $M>0$. Let now
$$ Z_\varepsilon=\{v\in Y\mid \|v\|_{Y}\le 2C\varepsilon\}. $$
We want to prove that $\Phi\colon Z_\varepsilon\to Z_\varepsilon$, where we recall that 
$$ \|u_0\|_{L^1\cap H^1_\alpha(\mathbb R^2)}\le \varepsilon. $$
With this hypothesis, we know from \eqref{proof.gl.es.3} that for any $v\in Z_\varepsilon$
$$ \|\Phi(v)\|_{Y}\le C\left[\|u_0\|_{H^s_\alpha(\mathbb R^2)} + \|v\|_Y^\gamma\right]\le C\varepsilon(1+2^\gamma C^\gamma\varepsilon^{\gamma-1}). $$
So, if we consider $\varepsilon>0$ sufficiently small such that 
$$ 1+2^\gamma C^\gamma\varepsilon^{\gamma-1}\le 2, $$
then $\Phi\colon Z_\varepsilon\to Z_\varepsilon$. Moreover, from \eqref{proof.gl.es.4}
$$ \|\Phi(v_1)-\Phi(v_2)\|_{Y}\le  M\left(\|v_1\|_{Y}^{\gamma-1} + \|v_2\|_{Y}^{\gamma-1}\right)\|v_1-v_2\|_{Y}\le 2^{\gamma}C^{\gamma-1}M\varepsilon^{\gamma-1}\|v_1-v_2\|_Y, $$
for any $v_1,v_2\in Z_\varepsilon$, so choosing $\varepsilon>0$ sufficiently small $\Phi\colon Z_\varepsilon\to Z_\varepsilon$ is a contraction. So, there is a unique solution in $Z_\varepsilon$ and, similarly as for the local existence, it can be seen that the solution is unique in $Y$. As for the local existence, it can be proved that 
$$ u\in L^rH^{s+1}\cap L^pL^q, $$
for any $s\in\left(0,\frac{2}{r}\right)$ and for any $p,q\in(1,\infty)$ such that
$$ \frac{1}{p} + \frac{1}{q} < 1. $$
This proves Theorem \ref{t.ex.gl.m.}. Finally, thanks to Proposition \ref{p.ex.Lagr.sys.}, we can find $\rho(t)$ which satisfies
$$ \|\rho\|_{L^1\cap L^r}\lesssim \|a\cdot \nabla(|u|^\gamma)\|_{L^1L^\frac{1}{1-\varepsilon}\cap L^rL^\frac{2}{1+\varepsilon}}\lesssim \|u\|_Y^\gamma, $$
where we used \eqref{gl.es.1} and \eqref{gl.es.2}.
\end{proof}

\subsection{Decay Estimates and Proof of Theorem \ref{t.decay.m.}}

We have just proved in the previous section Theorem \ref{t.ex.gl.m.2}, which ensures the existence of a solution $(u,\rho)$ for  
$$ \left\{\begin{array}{ll}
    (\partial_t-\Delta_\alpha)u + \rho\psi_\alpha=a\cdot\nabla(|u|^\gamma) & \mathbb R_+\times\mathbb R^2 \\
    P_{ac}u=u & \mathbb R_+\times \mathbb R^2 \\
    u(0)=P_{ac}u_0 & \mathbb R^2.
\end{array}\right. $$
Finally, to prove Theorem \ref{t.decay.m.}, it remains to show that such a solution decays in time. Let $h_1,h_2\in(1,\infty)$ and $\delta>0$, then we define the space
\begin{equation}\label{def.D}
  D_{h_1,h_2,\delta}=\left\{v\in L^\infty\left(\mathbb R_+;H^1_\alpha\left(\mathbb R^2\right)\right)\:\Big|\: |v|_{D_{h_1,h_2,\delta}}<+\infty\right\},  
\end{equation}
where
$$ |v|_{D_{h_1,h_2,\delta}}=\sup_{t\ge 1}t^{1-\frac{1}{h_1}-\delta}\|v(t)\|_{L^{h_1}(\mathbb R^2)} + \sup_{t\ge 1}t^{\frac{3}{2}-\frac{1}{h_2}-\delta}\|\nabla v(t)\|_{L^{h_2}(\mathbb R^2)}. $$
To be noticed that, for any $v\in D_{h_1,h_2,\delta}$, it holds 
$$ \|v(t)\|_{L^{h_1}(\mathbb R^2)}\le |v|_{D_{h_1,h_2,\delta}}t^{-1+\frac{1}{h_1}+\delta},\quad \|\nabla v(t)\|_{L^{h_2}(\mathbb R^2)}\le |v|_{D_{h_1,h_2,\delta}}t^{-\frac{3}{2}+\frac{1}{h_2}+\delta}. $$
In particular, if we manage to the prove that our solutions belong to $D_{h_1,h_2,\delta}$ for $\delta=\delta(h_1,h_2)>0$ sufficiently small, then we get also a polynomial decay in time for the functions. In order to do so, let us consider the Duhamel formula correspondent to the solution of the system \eqref{m.sys.}:
$$ u(t)=e^{P_{ac}\Delta_\alpha t}u_0 + \int_0^t e^{P_{ac}\Delta_\alpha (t-\tau)} a\cdot \nabla(|u(\tau)|^\gamma)d\tau. $$
Our goal is to bound each term of the Duhamel formula in $|\cdot|_{D_{h_1,h_2,\delta}}$:
\begin{lem}\label{l.dec.es.1}
Let $u_0\in L^1\cap L^2$, then for any $q\in(1,\infty)$ and $p\in(1,2)$ there exists $\delta_0=\delta_0(p,q)$ such that
$$ \left|e^{P_{ac}\Delta_\alpha t}u_0\right|_{D_{q,p,\delta}}\lesssim \|u_0\|_{L^1\cap L^2} \quad \forall \delta\in(0,\delta_0). $$
\end{lem}
\begin{proof}\hfill\\
By Theorem \ref{t.expDelta alpha}
$$ \|e^{P_{ac}\Delta_\alpha t}u_0\|_{L^q(\mathbb R^2)}\lesssim t^{-1+\delta+\frac{1}{q}}\|u_0\|_{L^\frac{1}{1-\delta}(\mathbb R^2)}\lesssim t^{-1+\delta+\frac{1}{q}}\|u_0\|_{L^1\cap L^2(\mathbb R^2)}. $$
Thanks to Theorem \ref{t.sem-grad.es.}, we can do the same to bound the gradient.
\end{proof}
For the inhomogeneous term, we need a technical lemma:
\begin{lem}\label{l.tec.}
Let $\alpha<1$ and $\beta\in\mathbb R$, then 
$$ \int_1^t (t-\tau)^{-\alpha}\tau^{-\beta}d\tau\lesssim t^{1-\alpha-\beta} \quad \forall t\ge 1. $$
\end{lem}
\begin{proof}\hfill\\
If $t\le 2$
$$ \int_1^t (t-\tau)^{-\alpha}\tau^{-\beta}d\tau \le \int_1^t (t-\tau)^{-\alpha}d\tau=\frac{1}{1-\alpha}(t-1)^{1-\alpha}\lesssim 1, $$
which is sufficient since $1\le t\le 2$. Let us suppose now $t>2$, then we split the integral in two pieces:
$$ \int_1^t (t-\tau)^{-\alpha}\tau^{-\beta}d\tau= \int_1^{t/2} (t-\tau)^{-\alpha}\tau^{-\beta}d\tau + \int_{t/2}^t (t-\tau)^{-\alpha}\tau^{-\beta}d\tau. $$
So
$$ \int_1^{t/2}(t-\tau)^{-\alpha}\tau^{-\beta}d\tau\lesssim t^{-\alpha}\int_1^{t/2}\tau^{-\beta}d\tau\simeq t^{1-\alpha-\beta}, $$
$$ \int_{t/2}^t(t-\tau)^{-\alpha}\tau^{-\beta}d\tau \lesssim t^{-\beta}\int_{t/2}^t(t-\tau)^{-\alpha}d\tau\simeq t^{1-\alpha-\beta}. $$
\end{proof}
\begin{lem}\label{l.dec.es.2}
Let $\gamma\ge 2$, $a\in\mathbb R^2$ and let $q\in(1,\infty)$,  $p\in(1,2)$ such that there are $\theta_1,\theta_2\in[0,1]$ with
$$ \theta_1+\theta_2>\frac{3}{2},\quad \max\left\{\frac{1}{p};\frac{1}{q}\right\}<\frac{\theta_1}{q}+\frac{\theta_2}{p}+\frac{1-\theta_2}{2}<1, $$
then there is $\delta_0>0$ sufficiently small such that, for any $\delta<\delta_0$ and for any $v\in D_{q,p,\delta}$
it holds
$$ \left|\int_0^te^{P_{ac}\Delta_\alpha (t-\tau)}a\cdot \nabla (|v(\tau)|^\gamma)d\tau\right|_{D_{q,p,\delta}}\lesssim \|v\|_{L^\infty H^1_\alpha}^{\gamma} + |v|_{D_{q,p,\delta}}^\gamma. $$
\end{lem}
\begin{proof}\hfill\\
We recall from the definition of $D_{q,p,\delta}$ in \eqref{def.D}, that $t\ge 1$. Then
$$ \int_0^t e^{P_{ac}\Delta_\alpha (t-\tau)} a\cdot \nabla(|v(\tau)|^\gamma)d\tau= I_1(t) + I_2(t), $$
where
$$ I_1(t) = \int_0^1 e^{P_{ac}\Delta_\alpha (t-\tau)} a\cdot\nabla(|v(\tau)|^\gamma)d\tau, $$
$$ I_2(t)= \int_1^t e^{P_{ac}\Delta_\alpha (t-\tau)} a\cdot \nabla(|v(\tau)|^\gamma)d\tau. $$
Let us start from $I_1(t)$: from Theorem \ref{t.expDelta alpha}
$$ \|I_1(t)\|_{L^q(\mathbb R^2)}\lesssim \int_0^1 (t-\tau)^{-1+\delta + \frac{1}{q}}\|v(\tau)^{\gamma-1}\nabla v(\tau)\|_{L^\frac{1}{1-\delta}(\mathbb R^2)}d\tau $$
$$ \le \int_0^1 (t-\tau)^{-1+\delta + \frac{1}{q}}\|v(\tau)\|_{L^\frac{2(\gamma-1)}{1-2\delta}(\mathbb R^2)}^{\gamma-1}\|\nabla v(\tau)\|_{L^2(\mathbb R^2)}d\tau \lesssim \|v\|_{L^\infty H^1_\alpha}^\gamma\int_0^1(t-\tau)^{-1+\delta+\frac{1}{q}}d\tau,  $$
where in the last inequality we used the Sobolev embedding $H^1_\alpha(\mathbb R^2)\hookrightarrow L^r(\mathbb R^2) $ for any $r\in[2,\infty)$. So
$$ \sup_{t\ge 1}\|I_1(t)\|_{L^q(\mathbb R^2)}\lesssim \|v\|_{L^\infty H^1_\alpha}^\gamma \sup_{t\ge 1}t^{1-\frac{1}{q}-\delta}\left[t^{\frac{1}{q}+\delta}-(t-1)^{\frac{1}{q}+\delta}\right]\lesssim \|v\|_{L^\infty H^1_\alpha}^\gamma. $$
Similarly, by Theorem \ref{t.sem-grad.es.} it holds
$$ \|\nabla I_1(t)\|_{L^p(\mathbb R^2)}\le \int_0^1 (t-\tau)^{-\frac{3}{2}+\delta + \frac{1}{p}}\|v(\tau)^{\gamma-1}\nabla v(\tau)\|_{L^\frac{1}{1-\delta}(\mathbb R^2)}d\tau $$
$$ \lesssim \|v\|_{L^\infty H^1_\alpha}^\gamma\int_0^1(t-\tau)^{-\frac{3}{2}+\delta+\frac{1}{p}}d\tau.  $$
So, since $\frac{3}{2}-\delta-\frac{1}{p}<1$ for $p<2$, we get
$$ |I_1|_{D_{q,p,\delta} }\lesssim \|v\|_{L^\infty H^1_\alpha}^\gamma. $$
Let us pass to $I_2(t)$: thanks to the condition on $p,q$, we can find $\theta_1,\theta_2$ and $h>1$ sufficiently large such that 
$$ \|I_2(t)\|_{L^q}\lesssim \int_1^t (t-\tau)^{-\frac{1}{\ell}+\frac{1}{q}}\|v(\tau)^{\gamma-1}\nabla v(\tau)\|_{L^\ell}d\tau,  $$
for 
$$ \frac{1}{\ell}=\frac{\theta_1}{q} + \frac{\theta_2}{p}+\frac{1-\theta_2}{2} + \frac{1}{h}\in\left(\frac{1}{q},1\right). $$
We notice that 
$$ u(\tau)^{\gamma-1}\nabla u(\tau)=u(\tau)^{\theta_1}u(\tau)^{\gamma-1-\theta_1}\nabla u(\tau)^{\theta_2}\nabla u(\tau)^{1-\theta_2}, $$
so by H\"older inequality and Sobolev embeddings
$$ \|u(\tau)^{\gamma-1}\nabla u(\tau)\|_{L^\ell(\mathbb R^2)}\lesssim \|u(\tau)\|_{L^q(\mathbb R^2)}^{\theta_1}\|u\|_{L^{\widetilde \ell}(\mathbb R^2)}^{\gamma-1-\theta_1}\|\nabla u(\tau)\|_{L^p}^{\theta_2}\|\nabla u(\tau)\|_{L^2}^{1-\theta_2} $$
$$ \le\tau^{-\theta_1\left(1-\frac{1}{q}+\delta\right)-\theta_2\left(\frac{3}{2}-\frac{1}{p}\right)}|u|_{D_{q,p,\delta}}^{\theta_1+\theta_2}\|u\|_{L^\infty H^1_\alpha}^{\gamma-\theta_1-\theta_2}. $$
Therefore, if we call
$$ \|u\|_Y=\|u\|_{L^\infty H^1_\alpha} + |u|_{D_{q,p,\delta}}, $$
we get
$$ \|I_2(t)\|_{L^q(\mathbb R^2)}\lesssim \|u\|_Y^\gamma\int_1^t(t-\tau)^{-\frac{1}{\ell}+\frac{1}{q}}\tau^{-\theta_1\left(1-\frac{1}{q}+\delta\right)-\theta_2\left(\frac{3}{2}-\frac{1}{p}-\delta\right)}d\tau. $$
So, by Lemma \ref{l.tec.}, we get 
$$ \|I_2(t)\|_{L^q}\lesssim t^{\frac{1}{2}-\frac{1}{h}+\frac{1}{q}-(1-\delta)(\theta_1+\theta_2)}\left[\|u\|_{L^\infty H^1_\alpha}^\gamma + |u|_{D_{q,p,\delta}}^\gamma\right]. $$
In particular, choosing $\delta$ sufficiently small, $h$ sufficiently large and thanks to the condition $\theta_1+\theta_2>\frac{3}{2}$, we get 
$$ \|I_2(t)\|_{L^q(\mathbb R^2)}\lesssim t^{-1+\frac{1}{q}+\delta}\left[\|u\|_{L^\infty H^1_\alpha}^\gamma + |u|_{D_{q,p,\delta}}^\gamma\right]. $$
Similarly, thanks to the hypothesis, we can find $\ell>1$ such that
$$ \frac{1}{\ell}=\frac{\theta_1}{q} + \frac{\theta_2}{p}+\frac{1-\theta_2}{2} + \frac{1}{h}\in\left(\frac{1}{p},1\right). $$
Then
$$ \|\nabla I_2(t)\|_{L^p(\mathbb R^2)}\lesssim \int_1^t(t-\tau)^{-\frac{1}{2}-\frac{1}{\ell}+\frac{1}{p}}\|u(\tau)^{\gamma-1}\nabla u(\tau)\|_{L^\ell(\mathbb R^2)}d\tau $$
$$ \lesssim \|u\|_Y^\gamma\int_1^t (t-\tau)^{-\frac{1}{2}-\frac{1}{\ell}+\frac{1}{p}}\tau^{-\theta_1\left(1-\frac{1}{q}+\delta\right)-\theta_2\left(\frac{3}{2}-\frac{1}{p}-\delta\right)}d\tau 
 $$
 $$ \lesssim t^{-\frac{1}{h}+\frac{1}{p}-(1-\delta)(\theta_1+\theta_2)}\|u\|_Y^\gamma\le t^{-\frac{3}{2}+\frac{1}{p}}\|u\|_Y^\gamma, $$
choosing $\delta>0$ sufficiently small and $h$ sufficiently large.
\end{proof}
We are finally ready to prove the decay in time of the solution:
\begin{proof}[Proof of Theorem \ref{t.decay.m.}]\hfill\\
Let as before
$$ X= L^\infty H^1_\alpha\cap L^r H^{s+1}_\alpha \cap L^{p_1}L^{q_1}\cap L^{p_2}L^{q_2}, $$
for $p_1,p_2\in (1,\infty)$ and $(q_1,q_2)\in[2,\infty)$ such that 
$$ \frac{1}{p_i} + \frac{1}{q_i} < 1\quad i=1,2, $$
to be chosen later. In this case, we consider
$$ Y=\left\{v\in X\:\Big|\: v(0)=v_0,\quad |v|_{D_{h_1,h_2,\varepsilon}}<+\infty\right\}, $$
endowed with the norm
$$ \|v\|_{Y}=\|v\|_X + |v|_{D_{h_1,h_2,\varepsilon}}, $$
where
$$ \|v\|_X= \|v\|_{L^\infty_TH^s_\alpha} + \|v\|_{L^2_TH^{s+1}_\alpha} + \|v\|_{L^{p_1}L^{q_1}} + \|v\|_{L^{p_2}L^{q_2}}. $$
Let $\Phi\colon X\to X$ where $\Phi(v)=u$ with $u$ solution of
$$ \left\{\begin{array}{ll}
    (\partial_t-P_{ac}\Delta_\alpha)u=P_{ac}(a\cdot\nabla) (|v|^\gamma) & \mathbb R_+\times\mathbb R^2 \\
    P_{ac}u=u & \mathbb R_+\times \mathbb R^2 \\
    u(0)=P_{ac}u_0 & \mathbb R^2. 
\end{array}\right. $$
Thanks to Lemma \ref{l.dec.es.1} and \ref{l.dec.es.2}
\begin{equation}\label{proof.dec.es.1}
    |\Phi(v)|_{D_{h_1,h_2,\varepsilon}}\lesssim \|u_0\|_{L^1\cap L^2(\mathbb R^2)} + \|v\|_Y^\gamma.
\end{equation}
So, using also \eqref{proof.gl.es.4}, we get
\begin{equation}\label{proof.gl.dec.2}
    \|\Phi(v)\|_{Y}\le C\left[\|u_0\|_{L^1\cap H^1_\alpha(\mathbb R^2)} + \|v\|^\gamma_{Y}\right],
\end{equation}
and similarly
\begin{equation}\label{proof.gl.dec.3}
    \|\Phi(v_1)-\Phi(v_2)\|_{Y}\le  M\left(\|v_1\|_{Y}^{\gamma-1} + \|v_2\|_{Y}^{\gamma-1}\right)\|v_1-v_2\|_{Y},
\end{equation}
for some $C,M>0$. As before, we deduce that $\Phi$ admits a fixed point $u$ in $Y$, which is consequently a solution for the system \eqref{sys.gl.}. Finally, as we did in the proof of Lemma \ref{l.dec.es.1}, it can be seen that 
$$ |\rho(t)|\le \|a\cdot\nabla|u(t)|^\gamma\|_{L^\frac{1}{1-\varepsilon}(\mathbb R^2)}\lesssim t^{-\theta_1\left(1-\frac{1}{h_1}-\varepsilon\right)-\theta_2\left(\frac{3}{2}-\frac{1}{h_2}-\varepsilon\right)}.
$$
We notice that 
$$ -\theta_1\left(1-\frac{1}{h_1}-\varepsilon\right)-\theta_2\left(\frac{3}{2}-\frac{1}{h_2}-\varepsilon\right) $$
$$ = -\frac{1}{2}+\frac{\theta_1}{h_1}+\frac{\theta_2}{h_2}+\frac{1-\theta_2}{2}-(1-\varepsilon)(\theta_1+\theta_2) <\frac{1}{2}-\theta_1-\theta_2 + o(\varepsilon), $$
where we used the conditions of $\theta_1,\theta_2,p,q$. In particular, since $\theta_1+\theta_2>\frac{3}{2}$, choosing $\varepsilon$ sufficiently small we conclude.
\end{proof}

\vspace{2mm}

\textbf{Acknowledgements:} The authors were partially supported by INDAM, GNAMPA group. D.B. is supported by the project E53D23005450006 “Nonlinear dispersive equations in presence of singularities” - funded by European Union - Next Generation EU within the PRIN 2022 program (D.D. 104 - 02/02/2022 Ministero dell’Università e della Ricerca). This manuscript reflects only the author’s views and opinions and the Ministry cannot be considered responsible for them.

\clearpage
\bibliographystyle{plain}
 \bibliography{BM}

\end{document}